\documentclass[11pt]{article}
\usepackage[T1]{fontenc}
\usepackage[utf8]{inputenc}
\usepackage{geometry}
\usepackage{enumitem}
\usepackage{lmodern}
\usepackage{mathrsfs}
\usepackage{mathtools}
\usepackage{amsmath}
\usepackage{amsthm}
\usepackage{amssymb}
\usepackage{microtype}
\usepackage[
colorlinks=true,
pdfpagemode=UseNone,
urlcolor=blue,
linkcolor=blue,
citecolor=blue,
breaklinks=true
]{hyperref}
\usepackage[
capitalize,
nosort
]{cleveref}
\theoremstyle{plain}
\newtheorem{thm}{Theorem}[section]
\newtheorem{prop}[thm]{Proposition}
\newtheorem{lem}[thm]{Lemma}
\newtheorem{cor}[thm]{Corollary}
\theoremstyle{definition}
\newtheorem{assumption}{Assumption}
\theoremstyle{remark}
\newtheorem{rem}[thm]{Remark}
\numberwithin{equation}{section}

\Crefname{assumption}{Assumption}{Assumptions}
\Crefname{prop}{Proposition}{Propositions}
\Crefname{lem}{Lemma}{Lemmas}
\Crefname{thm}{Theorem}{Theorems}
\newcommand*{\assuref}[1]{\cref{assu:#1}}

\crefdefaultlabelformat{#2\textup{#1}#3}
\crefformat{equation}{\textup{(#2#1#3)}}
\crefmultiformat{equation}{(#2#1#3)}%
{ and~(#2#1#3)}{, (#2#1#3)}{ and~(#2#1#3)}

\newcommand*{\email}[1]{{\href{mailto:#1}{\nolinkurl{#1}}}}
\newcommand*{\braket}[1]{\langle #1 \rangle}
\newcommand*{\supp}{\operatorname{supp}}
\newcommand*{\sgn}{\operatorname{sgn}}
\newcommand*{\Law}{\operatorname{Law}}
\newcommand*{\Uniform}{\operatorname{Uniform}}
\newcommand*{\R}{\mathbb{R}}
\newcommand*{\Z}{\mathbb{Z}}
\newcommand*{\N}{\mathbb{N}}
\renewcommand*{\P}{\mathbb{P}}
\newcommand*{\E}{\mathbb{E}}
\newcommand*{\T}{\mathbb{T}}
\newcommand*{\e}{\mathrm{e}}
\newcommand*{\p}{\mathrm{p}}
\newcommand*{\dif}{d}
\newcommand*{\eps}{\varepsilon}
\newcommand*{\one}{\mathbf{1}}
\newcommand*{\anon}{\,\cdot\,}
\newcommand*{\hL}{\hat{L}}
\newcommand*{\smoothing}{\mathfrak{S}_\ell}

\title{Invariant measures for stochastic conservation laws on the line}
\author{
  Theodore D. Drivas\thanks{Mathematics Department, Stony Brook University, Stony Brook, NY 11794 USA. \email{tdrivas@math.stonybrook.edu}}
  \thanks{School of Mathematics, Institute for Advanced Study, 1 Einstein Dr., Princeton, NJ 08540 USA. \email{tdrivas@ias.edu}}\and
  Alexander Dunlap\thanks{Department of Mathematics, Courant Institute of Mathematical Sciences, New York University, New York, NY 10012 USA. \email{alexander.dunlap@cims.nyu.edu}}\and
  Cole Graham\thanks{Division of Applied Mathematics, Brown University, Providence, RI 02912 USA. \email{cole_graham@brown.edu}}\and
  Joonhyun La\thanks{Department of Mathematics, Imperial College London, London SW7 2AZ United Kingdom. \email{jla@ic.ac.uk}}\and
  Lenya Ryzhik\thanks{Department of Mathematics, Stanford University, Stanford, CA 94305 USA. \email{ryzhik@stanford.edu}}
}

\begin{document}
\maketitle

\begin{abstract}
  We consider a stochastic conservation law on the line with solution-dependent diffusivity, a super-linear, sub-quadratic Hamiltonian, and smooth, spatially-homogeneous kick-type random forcing.
  We show that this Markov process admits a unique ergodic spatially-homogeneous invariant measure for each mean in a non-explicit unbounded set.
  This generalizes previous work on the stochastic Burgers equation.
\end{abstract}

\section{Introduction}

We consider the stochastic conservation law
\begin{equation}
  \label{eq:uPDE}
  \partial_{t}u = \partial_{x}\big[\kappa(u)\partial_{x}u - H(u) + V(t,x)\big]
  \quad \text{for } t \in \R_+,\, x \in \R.
\end{equation}
Here, $\kappa(u)$ is a H\"older continuous nonlinear diffusivity bounded from above and below, 
$H(u)$ is a sub-quadratic, super-linear Hamiltonian, and $V(t,x)$ is a random, space-stationary noise that is smooth in space and ``kick-type'' in time.
We defer the precise assumptions on $\kappa$, $H$, and $V$ to \cref{assu:kappa,assu:H,assu:V}, respectively, in \cref{subsec:Assumptions} below.

The stochastic Burgers equation is an important special case of \cref{eq:uPDE},
corresponding to constant~$\kappa$ and $H(u)=u^{2}/2$.
Ergodic properties of this equation  on the whole line 
and with various forms of the random noise $V(t,x)$ have been studied    
in several previous works~\cite{BC21,BL18,BL19,DGR21,DR21}. 
The inviscid case $\kappa\equiv0$ on the line was studied earlier in \cite{Bak16,BCK14}.
We refer to the survey \cite{BK18} for a review of the rich literature,  especially on the torus,  and for many interesting perspectives
on this problem. 
A key result is that the stochastic Burgers equation admits a unique global ergodic spacetime-stationary solution for every mean $a\in\R$ (\cite[Theorem~3.1]{BL19} and \cite[Theorem 1.2]{DGR21}).
In the present paper, motivated  by some problems posed in \cite{BK18},
we partially extend this claim to more general choices of $\kappa(u)$ 
and $H(u)$.
The main result is that under appropriate assumptions on $\kappa$, $H$, and $V$, there is an unbounded set of means for which an ergodic spacetime-stationary solution of \cref{eq:uPDE} exists.
Moreover, each mean admits at most one ergodic spacetime-stationary law.

Most approaches to the stochastic Burgers equation interpret it in terms of directed polymers using the Cole--Hopf transform, which is specific to the Burgers case.
In \cite[§3]{BK18}, the authors propose using a generalization of directed polymers to analyze more general Hamiltonians.
However, to the best of our knowledge, this direction has yet to bear fruit except for some results in \cite{GIKP05}  in the compact setting.
In the present paper we use, instead,
the methods of \cite{DGR21}, which are PDE-based and less strongly
tied to the exact algebraic form of the stochastic Burgers equation.
We note that several previous works have tackled general Hamiltonians on the torus using PDE methods; see \cite{Bor13sharp,Bor14Review,Bor16multi,DV15,DS05}.
The problem on the line is quite different, mostly due to the lack of the Poincar\'e inequality; the main challenge in the present setting is the noncompactness of the domain.

The choice of kick-type $V(t,x)$ is motivated
by relative simplicity of exposition. Because we are interested in the long-time/stationary
behavior of \cref{eq:uPDE}, we do not expect the precise form of $V(t,x)$ to be important from a physical perspective.
However, if $\kappa$ is non-constant, there do seem to be technical challenges involved in extending our results to  white in time~$V(t,x)$, which is the setting considered in \cite{DGR21}.

We note that our results here are
weaker than those in \cite{DGR21} for the stochastic Burgers equation.
We do not show the existence of an invariant measure for  every mean
but only for means in a non-explicit unbounded set,
and we do not prove any stability or convergence result. Previous
proofs of these results for the stochastic Burgers equation rely on the shear-invariance of the stochastic Burgers equation; see \cite[p. R93]{BK18} for discussion.
In particular, the shear-invariance yields the exact form of the so-called shape function. There
is no shear invariance when the Hamiltonian is non-quadratic or when
the diffusivity is non-constant. Similar problems in the discrete setting lacking shear-invariance were studied in \cite{JRA20,JRAS21}.
There, special noise distributions yield an exactly-integrable structure that determines the shape function, compensating for the lack
of shear-invariance.
In addition, \cite{JRA20,JRAS21} proved various results about analogues of invariant measures conditional on certain understanding of the shape function.

We now set up the main result of this paper.
Because the law of the ``kick'' noise $V$ is $1$-periodic in time, we can view the evolution as a discrete time-homogeneous Markov process.
For the state space, we use a weighted function space $\mathcal{X}$ defined in \cref{subsec:functionspaces}.
Let~$\mathscr{P}(\mathcal{X})$ denote the space of Borel probability measures on the Polish space $\mathcal{X}$.
Given a measure~$\nu \in \mathscr{P}(\mathcal{X})$ and~$k \in \Z_{\geq 0}$, define $\mathcal{Q}_k^* \nu \coloneqq \Law u(k-, \anon)$, where $u$ solves \cref{eq:uPDE} with initial condition~$u(0-, \anon) \sim \nu$ that is independent of the noise $V$.
Then $(\mathcal{Q}_k)_{k = 0}^\infty$ forms a discrete time-homogeneous Markov semigroup.

We next define space- and time-invariant measures.
Let $\mathscr{P}_{\R}(\mathcal{X})$ denote the set of measures~$\nu \in \mathscr{P}(\mathcal{X})$ that are invariant under the translation action of $\R$ on $\mathcal{X}$.
Similarly, define
\begin{equation*}
  \overline{\mathscr{P}}(\mathcal{X}) = \big\{\nu \in \mathscr{P}(\mathcal{X}) \mid \mathcal{Q}_1^* \nu = \nu\big\}.
\end{equation*}
Note that, by the semigroup property, $\mathcal{Q}_k^* \nu = \nu$ for all $k \in \Z_{\geq 0}$ and all $\nu \in \overline{\mathscr{P}}(\mathcal{X})$.
Also, let
\begin{equation*}
  \overline{\mathscr{P}}_{\R}(\mathcal{X}) \coloneqq \mathscr{P}_{\R}(\mathcal{X}) \cap \overline{\mathscr{P}}(\mathcal{X}).
\end{equation*}
Finally, let $\overline{\mathscr{P}}_{\R}^{\mathrm{e}}(\mathcal{X})$ denote the set of \emph{extremal} measures in $\overline{\mathscr{P}}_{\R}$.
That is, the set of measures in $\overline{\mathscr{P}}_{\R}(\mathcal{X})$ that cannot be written as nontrivial convex combinations of other elements of~$\overline{\mathscr{P}}_{\R}(\mathcal{X})$.
By a standard result of ergodic theory, these are precisely the ergodic measures of the Markov semigroup.
We define and discuss ergodicity in \cref{subsec:invmeasures}.

With this notation, we can summarize the main result.
\begin{thm}
  \label{thm:main}
  There exist infinitely many $\nu\in\overline{\mathscr{P}}_{\R}^{\mathrm{e}}(\mathcal{X})$ such that if $v \sim \nu$, then $\E H(v)< \infty$ and $\E(\partial_{x}v)^{2}<\infty.$
  Moreover, for each $a \in \R$ there is at most one such measure with $\E v = a$.
\end{thm}
\begin{rem}
  This theorem follows from somewhat stronger results stated in \cref{subsec:results} below.
  In particular, it is a consequence of \cref{thm:existence} and \cref{cor:uniqueness}.
\end{rem}
We state our precise assumptions and results in \cref{sec:setting}.
In \cref{sec:wellposedness}, we establish a well-posedness theory for \cref{eq:uPDE} in weighted spaces.
We construct invariant measures in \cref{sec:existence} and prove uniqueness in \cref{sec:uniqueness}.

\subsection{Acknowledgments}

We thank  Yuri Bakhtin and Elena Kosygina for helpful discussions, and Scott Armstrong for telling us about the supersolutions used in the proof of \cref{lem:fkbd}. We also thank an anonymous referee for a careful reading of the manuscript and several helpful comments.
T.D. was partially supported by NSF grant DMS-2106233 and the Charles Simonyi Endowment at the Institute for Advanced Study. A.D. and C.G. were partially supported by the NSF Mathematical Sciences Postdoctoral Fellowship program under grants DMS-2002118 and DMS-2103383, respectively.
L.R. was partially supported by NSF grant DMS-1910023 and ONR grant N00014-17-1-2145.

\section{Setting and main results}
\label{sec:setting}

In this section, we clarify the setting of the problem, impose precise hypotheses on the terms of the PDE \cref{eq:uPDE}, and state the main results.

\subsection{Function spaces\label{subsec:functionspaces}}

We study solutions in weighted function spaces similar to those in \cite{DGR21}.
A weight is a positive function $w \colon \R\to\R_{>0}$. 
Given such weight $w$, we let $\mathcal{C}_{w}$ denote the Banach space of continuous functions $f \colon \R\to\R$ such that the
norm
\begin{equation*}
  \|f\|_{\mathcal{C}_{w}} \coloneqq \sup_{x\in\R}\frac{|f(x)|}{w(x)}
\end{equation*}
is finite.
Given $\alpha\in(0,1)$, we also define
\begin{equation*}
  \|f\|_{\mathcal{C}_{w}^{\alpha}} \coloneqq \|f\|_{\mathcal{C}_{w}}+\sup_{0<|x-y|\le1}\frac{|f(x)-f(y)|}{w(x)|x-y|^{\alpha}}.
\end{equation*}
We often use the weights $\p_{\ell}(x)=\langle x\rangle^{\ell}$, where $\langle x\rangle=\sqrt{4+x^{2}}$ and $\ell > 0$.
For each $m \geq 0$, we define the Fréchet space
\begin{equation*}
  \mathcal{X}_{m} \coloneqq \bigcap_{\ell>m}\mathcal{C}_{\p_{\ell}}
\end{equation*}
equipped with the Fréchet topology generated by the norms of $\{\mathcal{C}_{\p_{\ell}}\}_{\ell > m}$.
Throughout the paper, we distinguish the weighted space $\mathcal{X} \coloneqq \mathcal{X}_{\frac{2}{2 + q}}$, where $q \in (1, 2]$ is defined in \assuref{H} below. 
We also sometimes use weighted $L^p$ spaces for $p \in (1, \infty)$, defined by the norms
\begin{equation*}
  \|f\|_{L^p_w} = \left(\int_\R \left(\frac{|f(x)|}{w(x)}\right)^p\,\dif x\right)^{1/p}.
\end{equation*}

\subsection{Assumptions\label{subsec:Assumptions}}

We now state our assumptions on $\kappa$, $H$, and $V$.
\begin{assumption}[Regarding $\kappa$]
  \label{assu:kappa}
  \mbox{}
  \begin{enumerate}[label = \textup{(\roman*)}, leftmargin = 3.5em, labelsep = 1em, itemsep= 1ex, topsep = 1ex]
  \item
    There exists $\kappa_{0} \in (0, 1]$ such that
    \begin{equation}
      \label{eq:alliptic}
      \kappa_{0}\le\kappa(u)\le\kappa_{0}^{-1}\quad\text{for all }u\in\R.
    \end{equation}

  \item
    There exist $\alpha_\kappa \in (1/2, 1)$ and $\beta_{\kappa} \in (0, 1)$ such that $\|\kappa\|_{\mathcal{C}^{\alpha_\kappa}} < \infty$
    and $\kappa \in \mathcal{C}_{\mathrm{loc}}^{1,\beta_\kappa}$.

  \item
    There exists $C_\kappa <+\infty$ such that
  \begin{equation}
    \label{sep2218}
    |\kappa'(u)| \le C_\kappa(1+|u|)\quad \text{for all } u\in\R.
  \end{equation}
  \end{enumerate}
\end{assumption}
\begin{assumption}[Regarding $H$]
  \label{assu:H}
  \mbox{}
  \begin{enumerate}[label = \textup{(\roman*)}, leftmargin = 3.5em, labelsep = 1em, itemsep= 1ex, topsep = 1ex]
  \item
    The Hamiltonian $H$ is in $\mathcal{C}_{\mathrm{loc}}^{1,\alpha_H}$ for some $\alpha_H \in (0, 1)$.

  \item
    There exist constants $\lambda$, $c_1$, $c_2$, $C_H > 0$ and $q\in(1,2]$ such that
    \begin{equation}
      \label{eq:Hbd}
      c_1|u|^{q}-c_1^{-1}\le H(u)\le\lambda\kappa_{0}u^{2}+c_2 \quad \text{for all } u \in \R
    \end{equation}
    and
    \begin{equation}
      \label{eq:Hderivbd}
      |H'(u)|\le C_H(1+|u|)^{q/2} \quad \text{for all } u \in \R.
    \end{equation}
  \end{enumerate}
\end{assumption}
Throughout the paper, we will work on a probability space $(\Omega,\mathcal{F},\P)$.
We assume that it is large enough for all of the random variables that we define.
\begin{assumption}[Regarding $V$]
  \label{assu:V}
  The random distribution $V$ can be written as
  \begin{equation*}
    V(t,x)=\sum_{s\in\Z}\mathsf{V}_{s}(x)\delta(t-s),
  \end{equation*}
  where $(\mathsf{V}_{s})_{s\in\Z}$ is an iid family of space-translation-invariant random functions on $\R$ satisfying the following properties:
  \begin{enumerate}
   \item For each $s\in\Z$, we have $\partial_x \mathsf{V}_s\in\mathcal{X}_0$ almost surely.
   \item We have $\E[\mathsf{V}_s(x)^2]<\infty$ and $\E[(\partial_x\mathsf{V}_s(x))^2]<\infty$.
   \item If we define
   \begin{equation}
    \overline{\mathsf{V}}_s(j) = \inf_{x\in[j,j+1]}\mathsf{V}_s(x),\label{eq:Vmin}
   \end{equation}
   then \begin{equation}\E\e^{-\lambda\overline{\mathsf{V}}_s(j)}<\infty,\label{eq:vbarexpmoment}\end{equation}with $\lambda$ as in \cref{assu:H}.% and
  \end{enumerate}
%such that $\{\mathsf{V}_s(x)\}_{s\in\Z,x\in\R}$ is a Gaussian process and $\partial_x \mathsf{V}_s$ takes values in $\mathcal{X}_0$.
\end{assumption}
We note that (by, e.g.,\ the Fernique and Borell--TIS inequalities) smooth, rapidly decorrelating Gaussian random fields satisfy \cref{assu:V}.

\subsection{Invariant measures\label{subsec:invmeasures}}

Let $\Psi_{s}$ be the time-$s$ map of the  unforced dynamics
\begin{equation}
  \label{eq:uPDE-unforced}
  \partial_{t}u = \partial_{x}\big[\kappa(u)\partial_{x}u-H(u)\big].
\end{equation}
That is, if $u$ solves \cref{eq:uPDE-unforced} with $u(0,x)=v(x)$, then
\begin{equation}
  \label{eq:Psisdef}
  \Psi_{s}(v)(x) \coloneqq u(s,x) \quad \text{for } s \geq 0, \, x \in \R.
\end{equation}
In \cref{sec:wellposedness} below, we show that $\Psi_{s}$ is a continuous function from $\mathcal{X}$ to itself for each $s \geq 0$.

Given $s \geq t \geq 0$, if $u$ satisfies the stochastic PDE \cref{eq:uPDE} with $u(t-,x)=v(x)$, we define
\begin{equation*}
  \Phi_{t,s}(v)(x)=u(s-,x).
\end{equation*}
Note that $\Phi$ satisfies the co-cycle property: $\Phi_{t,t} = \operatorname{id}$ and $\Phi_{s,r}\circ\Phi_{t,s}=\Phi_{t,r}$.
Also, if $k\in\Z_{\geq 0}$ and $s\in (0,1]$,
\begin{equation*}
  \Phi_{t,k+s}(v) = \Psi_s\big(\Phi_{t,k}(v)+\partial_x\mathsf{V}_k\big).
\end{equation*}
We also extend $\Psi_{s}$ to act entry-wise on the product space $\mathcal{X}^{N}$:
\begin{equation*}
  \Psi_{s}(v_{1},\ldots,v_{N}) \coloneqq (\Psi_{s}(v_{1}),\ldots,\Psi_{s}(v_{N})).
\end{equation*}
Similarly, we let
\begin{equation*}
  \Phi_{t,s}(v_{1},\ldots,v_{N}) \coloneqq (\Phi_{t,s}(v_{1}),\ldots,\Phi_{t,s}(v_{N})).
\end{equation*}
We emphasize that the initial conditions $(v_{1},\ldots,v_{N})$ evolve subject to the \emph{same} realization of the noise process in the latter notation.

In the introduction, we framed the system in discrete time.
However, we frequently rely on continuous-time methods to control our solutions.
We therefore augment our state space to obtain a  continuous time-homogeneous Markov process,
by keeping track of the time since the last kick.
Given $t\in\R$, let
$\overline{t}$ denote its fractional part, i.e. the image of $t$ under the quotient map $\R\to\R/\Z$.
Set
\begin{equation*}
  \mathcal{A}_{N}=\mathcal{X}^{N}\times(\R/\Z),
\end{equation*}
and, given bounded measurable $F \colon \mathcal{A}_N \to \R$ and $t \geq 0$, define a bounded measurable map
\begin{equation*}
  \mathcal{P}_t F \colon \mathcal{A}_N \to \R
\end{equation*}
by
\begin{equation*}
  (\mathcal{P}_tF)(v_1, \ldots, v_N, \theta) = \E F\big(\Phi_{\theta, \theta + t}(v_1), \ldots, \Phi_{\theta, \theta + t}(v_N), \overline{\theta + t}\big).
\end{equation*}
Then $(\mathcal{P}_{t})_{t\ge0}$ forms a time-homogeneous Markov semigroup.
We define its dual action on~$\mathscr{P}(\mathcal{A}_N)$ by
\begin{equation}
  \label{eq:Ptdef}
  \mathcal{P}_t^*\mu = \Law\big(\Phi_{\theta, \theta + t}(v_1), \ldots, \Phi_{\theta, \theta + t}(v_N), \overline{\theta + t}\big)
\end{equation}
for $\mu \in \mathscr{P}(\mathcal{A}_N)$ and $(v_1, \ldots, v_N, \theta) \sim \mu$.

Next, let $G$ be the group $\R$ or $L \Z$ for some $L > 0$.
We let $\mathscr{P}_{G}(\mathcal{A}_{N})$ denote the set of measures in $\mathscr{P}(\mathcal{A}_{N})$ that are invariant under the action of $G$ on $\mathcal{A}_{N}$ by translation.
We likewise let $\overline{\mathscr{P}}(\mathcal{A}_{N})$ denote the set
of measures $\mu\in\mathscr{P}(\mathcal{A}_{N})$ such that $\mathcal{P}_{t}^{*}\mu=\mu$ for all $t\ge0$.
We observe that if $\mu\in\overline{\mathscr{P}}(\mathcal{A}_{N})$, then the marginal of $\mu$ on $\R/\Z$ is translation-invariant, and hence uniform.
Finally, we again define
\begin{equation}
  \label{eq:PGbar}
  \overline{\mathscr{P}}_{G}(\mathcal{A}_{N})=\mathscr{P}_{G}(\mathcal{A}_{N})\cap\overline{\mathscr{P}}(\mathcal{A}_{N})
\end{equation}
and we let $\overline{\mathscr{P}}_{G}^{\mathrm{e}}(\mathcal{A}_{N})$ denote the set of extremal elements of $\overline{\mathscr{P}}_{G}(\mathcal{A}_{N})$.
\begin{rem}
  \label{rem:ergodic}
  As noted in the introduction, extremality is equivalent to the notion of ergodicity.
  We say that a measure $\mu \in \overline{\mathscr{P}}_{G}(\mathcal{A}_{N})$ is  ergodic if the following holds.
  If $A \subset \mathcal{A}_N$ is a~$G$-invariant Borel set such that $\mathcal{P}_t \one_A = \one_A$ $\mu$-a.s. for all $t \geq 0$, then $\mu(A) = 0$ or $\mu(A) = 1$.
  The equivalence of extremality and ergodicity is a ubiquitous feature of ergodic theory.
  However, our setting falls outside traditional treatments of the subject because it involves both space and time actions.
  To handle our system, we first convert the Markov semigroup to a deterministic dynamical system following the approach of \cite[Section 4]{H08}.
  Then the equivalence of extremality and ergodicity follows from Theorem~3.1 in \cite{Varadarajan}, which treats general group actions on measure spaces.
\end{rem}
\medskip

In a certain sense, time-invariant measures on $\mathcal{A}_N$ are equivalent to those on $\mathcal{X}^N$.
To see this, define a map $\Sigma \colon \mathscr{P}(\mathcal{A}_N) \to \mathscr{P}(\mathcal{X}^N)$ by $\Sigma(\mu) \coloneqq \Law(\Phi_{\theta, 1}(\mathbf{v}))$ for $(\mathbf{v}, \theta) \sim \mu$.
We claim that $\Sigma$ preserves time invariance.
Indeed, the iid nature of the kicks $\mathsf{V}$ and the cocycle property implies:
\begin{equation}
  \label{eq:cocycle}
  \mathcal{Q}_1^*\Sigma(\mu) = \Law(\Phi_{1,2} \circ \Phi_{\theta, 1}(\mathbf{v})) = \Law(\Phi_{\theta, 2}(\mathbf{v})) = \Law(\Phi_{\theta+1, 2} \circ \Phi_{\theta, \theta + 1}(\mathbf{v})).
\end{equation}
Now if $\mu \in \overline{\mathscr{P}}(\mathcal{A}_N)$, we have
\begin{equation*}
  \Law(\mathbf{v}, \theta) = \mu = \mathcal{P}_1^*\mu = \Law(\Phi_{\theta,\theta+1}(\mathbf{v}), \overline{\theta + 1}) = \Law(\Phi_{\theta,\theta+1}(\mathbf{v}), \theta).
\end{equation*}
Thus \cref{eq:cocycle} yields
\begin{equation*}
  \mathcal{Q}_1^*\Sigma(\mu) = \Law(\Phi_{\theta+1, 2} \circ \Phi_{\theta, \theta + 1}(\mathbf{v})) = \Law(\Phi_{\theta+1, 2}(\mathbf{v})) = \Law(\Phi_{\theta, 1}(\mathbf{v})) = \Sigma(\mu),
\end{equation*}
and $\Sigma(\mu) \in \overline{\mathscr{P}}(\mathcal{X}^N)$ as claimed.

Next, we define a map $\Xi \colon \mathscr{P}(\mathcal{X}^N) \to \mathscr{P}(\mathcal{A}_N)$ as follows.
Given $\mathbf{v} \sim \nu \in \mathscr{P}(\mathcal{X}^N)$, take~$\vartheta \sim \Uniform(\R/\Z)$ independent of all else, and
let $\Xi(\nu) \coloneqq \Law(\Phi_{0, \vartheta}(\mathbf{v}), \vartheta)$.
Then we can easily check that $\Sigma \circ \Xi = \mathcal{Q}_1^*$ and, abusing notation, $(\Xi \circ \Sigma)(\mu) = \Law(\mathcal{P}_\vartheta^* \mu)$.
Therefore,~$\Sigma$ and $\Xi$ are inverses on the time-invariant measures $\overline{\mathscr{P}}(\mathcal{A}_N)$ and $\overline{\mathscr{P}}(\mathcal{X}^N)$.
Moreover, these maps commute with spatial translation and convex combination, so they define bijections~$\overline{\mathscr{P}}_G(\mathcal{A}_N) \leftrightarrow \overline{\mathscr{P}}_G(\mathcal{X}^N)$ and $\overline{\mathscr{P}}_G^{\mathrm{e}}(\mathcal{A}_N) \leftrightarrow \overline{\mathscr{P}}_G^{\mathrm{e}}(\mathcal{X}^N)$ as well.

\subsection{Main results}
\label{subsec:results}

We can now precisely state the main theorems.
First, \cref{eq:uPDE} admits infinitely many invariant measures.
\begin{thm}
  \label{thm:existence}
  There is a constant $C<+\infty$ depending only on the constant $\kappa_0$ in \cref{assu:kappa}, the constants $\lambda,c_1,c_2$ in \cref{assu:H}, and the law of $V$, such that the following holds.
  Let the group $G$ be $\R$ or $L\Z$ for some ${L>0}$ and let $X\sim\Uniform(\R/G)$.
  For each~$a \ge 0$, there exists $\mu_{a,G}\in\overline{\mathscr{P}}_{G}^{\mathrm{e}}(\mathcal{A}_1)$ such that if $(v, \theta)\sim\mu_{a,G}$, then
  \begin{equation*}
    \E H(v(X)),\, \E(\partial_{x}v(X))^{2}<\infty
  \end{equation*}
  and
  \begin{equation}
    \label{eq:mean-bd}
    - \frac{\braket{a}}{2}\leq \E v(X) - a \leq C \braket{a}^{\frac{1}{q - 1}}.
  \end{equation}
\end{thm}
\begin{rem}
  Let
  \begin{equation*}
    \mathscr{A} \coloneqq \left\{\E v(X) \in \R \mid \Law(v, \theta) \in \overline{\mathscr{P}}_{G}^{\mathrm{e}}(\mathcal{A}_1) \, \text{ and }\, \E H(v(X)),\, \E(\partial_{x}v(X))^{2}<\infty\right\}.
  \end{equation*}
  If $I_a$ denotes the interval in \cref{eq:mean-bd} centered on $a$, then \cref{thm:existence} ensures that $\mathscr{A} \cap I_a \neq \emptyset$.
  Since $I_{a_1} \cap I_{a_2} = \emptyset$ when $a_2 \gg a_1$, we see that $\mathscr{A} \cap \R_+$ is infinite.
  In particular, $\overline{\mathscr{P}}_{G}^{\mathrm{e}}(\mathcal{A}_1)$ is infinite and unbounded.
  Moreover, by replacing $u$ with $-u$, we also obtain infinitely many negative means.

  We observe, however, that it is possible to satisfy \cref{thm:existence} with a rather sparse set of means.
  For instance, the set
  \begin{equation*}
    \mathscr{A} = \left\{\pm \exp\left(cn (q - 1)^{-n}\right) \mid n \in \N\right\}
  \end{equation*}
  is compatible with \cref{eq:mean-bd} provided $c>0$ is sufficiently small.
\end{rem}
\begin{rem}
  In light of the bijection $\Sigma \colon \overline{\mathscr{P}}(\mathcal{A}_1) \to \overline{\mathscr{P}}(\mathcal{X})$ defined above, \cref{thm:existence} implies the existence portion of \cref{thm:main}.
\end{rem}
The next result says that  invariant ensembles of solutions evolving under the same noise can be invariantly coupled and are necessarily \emph{ordered}.
\begin{thm}
  \label{thm:stochordering}
  Fix $N_i \in \N$ for each $i \in \{1, 2\}$.
  Let the group $G$ be $\R$ or $L\Z$ for some ${L>0}$ and let $X\sim\Uniform(\R/G)$.
  Given $\mu_{i}\in\overline{\mathscr{P}}_{G}^{\mathrm{e}}(\mathcal{A}_{N_{i}})$, let $(\mathbf{v}_{i},\theta_{i})\sim\mu_{i}$ with $\mathbf{v}_{i}=(v_{i;1},\ldots,v_{i;N_{i}})$ for $i \in \{1,2\}$.
  Suppose
  \begin{equation*}
    \E H(v_{i;j}(X)), \, \E(\partial_{x}v_{i;j}(X))^{2}<\infty \quad \text{for all } i\in\{1,2\}, \, j \in \{1,\ldots,N_i\}.
  \end{equation*}
  Then there exists $\mu\in\overline{\mathscr{P}}_{G}^{\mathrm{e}}(\mathcal{A}_{N_{1}+N_{2}})$ such that if $((\mathbf{v}_{1},\mathbf{v}_{2}),\theta)\sim\mu$, then
  \begin{equation}
    \label{eq:isacoupling}
    \Law(\mathbf{v}_i, \theta) = \mu_i \quad \text{for each } i\in\{1,2\}
  \end{equation}
  and for each $j_{1}\in\{1,\ldots,N_{1}\}$ and $j_{2}\in\{1,\ldots,N_{2}\}$, there is a deterministic   $\chi_{j_1, j_2} \in \{0, \pm 1\}$ such that almost surely
  \begin{equation}
    \label{eq:ordering}
    \sgn\big(v_{1;j_{1}}(x)-v_{2;j_{2}}(x)\big) = \chi_{j_1, j_2} \quad \text{for all } x \in \R.
  \end{equation}
\end{thm}
This ordering implies uniqueness for suitably bounded invariant extremal measures.
\begin{cor}
  \label{cor:uniqueness}
  Fix $N \in \N$.
  Let $G$ be $\R$ or $L\Z$ for some ${L>0}$ and let $X\sim\Uniform(\R/G)$.
  Then for each $\mathbf{a} \in \R^N$, there exists at most one $\mu \in \overline{\mathscr{P}}_G^{\mathrm{e}}(\mathcal{A}_N)$ such that if $(\mathbf{v}, \theta) \sim \mu$, we have $\E \mathbf{v}(X) = \mathbf{a}$ and
  \begin{equation*}
    \E H(v_i(X)), \, \E (\partial_x v_i(X))^2 < \infty \quad \text{for each } i \in \{1, \ldots, N\}.
  \end{equation*}
\end{cor}
\begin{rem}
  Taking $N = 1$ and $G = \R$, \cref{cor:uniqueness} implies the uniqueness in \cref{thm:main}.
\end{rem}

\section{Well-posedness theory\label{sec:wellposedness}}

In this section, we study the well-posedness of the \emph{unforced} problem in weighted spaces that permit growth at infinity,
which places it outside traditional theory.
While~\cite{DGR21} established well-posedness for the Burgers equation in weighted spaces, the nonlinear diffusivity $\kappa(u)$
presents new difficulties, as we cannot rely on the theory of mild solutions (Duhamel's formula).
\begin{thm}
  \label{thm-sep2702}
  Fix $m\in(0,1)$.
  For each $s\ge0$, the map $\Psi_{s}$ defined in \cref{eq:Psisdef} is a continuous function from $\mathcal{X}_m$ to itself.
\end{thm}
This implies that the random dynamical system is Feller.
\begin{cor}
  \label{cor:feller}
  Fix $m\in(0,1)$.
  For any $s<t$, the map $\Phi_{s,t}:\mathcal{X}_m\to\mathcal{X}_m$ is continuous with probability $1$.
  Also, the semigroup $(\mathcal{P}_{t})_{t \geq 0}$ defined in \cref{eq:Ptdef} has the Feller property.
\end{cor}
The proof  of \cref{thm-sep2702} proceeds as follows.
First, we study the PDE on a torus of length $L$, and prove well-posedness for the periodized problem.
Next, we obtain weighted~$L^\infty$-estimates for the periodized problem that are uniform in $L$.
We can then take $L\to+\infty$ and obtain a solution in $\mathcal{X}_m$.
Finally, we prove the continuity of $\Psi_s$ using weighted $L^1$-stability.

\subsection{Global-in-time well-posedness of the periodized problem}

We first prove well-posedness on the torus $\T_L \coloneqq \R/L\Z$ for $L > 0$.
Consider the unforced problem \cref{eq:uPDE-unforced} on $\T_L$:
\begin{equation}
  \label{sep2210}
  \partial_{t}u = \partial_{x}\big[\kappa(u)\partial_{x}u-H(u)\big], \quad u(0, x)=u_0(x) \quad \text{for } x \in \T_L.
\end{equation}
To state the well-posedness result, we use the notation
\begin{equation}
  \label{eq:parabolic-cylinder}
  Q_r(t,x) = (t-r^2, t)\times B(x,r)
\end{equation}
for $t \geq r^2 > 0$ and $x \in \T_L$.
\begin{prop}
  \label{prop-sep2304}
  There exist $\alpha \in (0, 1)$ and $C_Q \colon \R_+ \to \R_+$ such that the following holds.
  For every $u_0 \in \mathcal{C}( \T_L )$, \cref{sep2210} admits a unique strong solution $u$ that is $\mathcal{C}_{\mathrm{loc}}^{2,\alpha}$ 
  in space and $\mathcal{C}_{\mathrm{loc}}^{1,\alpha/2}$ in time.
  Moreover, for every parabolic cylinder $Q_{2r}(t, x) \subset \R_+ \times \T_L$, we have
  \begin{equation}
    \label{sep2212}
    \|u\|_{\mathcal{C}^{\alpha}(Q_r(t, x))} \le C_Q(r) \|u\|_{L^\infty(Q_{2r}(t, x))}\left(1 + \|u\|_{L^\infty(Q_{2r}(t, x))}\right).
  \end{equation}
\end{prop}
This proposition will follow from \cref{lem:regularizedsoln,lem:rmreg,lem:weakisclassical,prop:prop0916} below.
We first construct a solution to a regularized version of \cref{sep2210}.
We then remove the regularization and show that the limit is a weak solution to~\cref{sep2210}.
A De Giorgi-type estimate allows us to upgrade regularity and show that the weak solution is actually strong.
We show uniqueness of the solution using $L^1$-contraction.

\subsubsection{The regularized problem}

Let $\phi\in\mathcal{C}_c^\infty(\R)$ be a nonnegative bump function with
\begin{equation*}
  \int_{\R} \phi = 1.  
\end{equation*}
Given $\ell>0$, define $\phi_\ell(x) \coloneqq \ell^{-1}\phi(\ell^{-1}x)$ and $\smoothing (u) \coloneqq \phi_\ell * u$, and 
consider the following regularized version of \cref{sep2210}:
\begin{equation}
  \label{sep2302}
  \partial_{t}u^\ell = \partial_{x}\big[\kappa\big(\smoothing (u^\ell)\big)\partial_{x}u^\ell - H(u^\ell)\big], \quad u^\ell(0, x)=u_0(x) \quad \text{for } x \in \T_L.
\end{equation}

\begin{lem}
  \label{lem:regularizedsoln}
  For each $\ell, T >0$, \cref{sep2302} admits a unique weak solution in
  \begin{equation*}
    \mathcal{B}\coloneqq L^2 (0, T; H^2(\T_L)) \cap L^\infty (0, T; H^1(\T_L))
  \end{equation*}
\end{lem}

\begin{proof}
  Fix $\ell, T > 0$.
  Let $u^{0,\ell}(t,x) \coloneqq u_0(x)$ and, for $n\ge 0$, let $u^{n+1,\ell}$ solve the linear problem
  \begin{equation}
    \label{eq:0811}
    \partial_t u^{n+1, \ell} = \partial_x \big[ \kappa\big( \smoothing ( u^{n, \ell})\big) \partial_x u^{n+1, \ell}\big] - H'(u^{n, \ell}) \partial_x u^{n+1, \ell}, \quad u^{n+1,\ell}(0,x)=u_0(x).
  \end{equation}
  The solution to \cref{eq:0811} exists and is unique by standard theory.
  The maximum principle implies that
  \begin{equation}
    \label{sep2214}
    \|u^{n,\ell}(t,\anon)\|_{L^\infty(\T_L)}  \le \|u_0\|_{L^\infty(\T_L)}\quad\text{for all } n\ge 0.
  \end{equation}
  Differentiating \cref{eq:0811}, multiplying by $\partial_x u^{n+1,\ell}$, and integrating  over $\T_L$ yields
  \begin{align}
    \frac 12  \frac{\dif}{\dif t} \| &\partial_x u^{n+1, \ell} \|_{L^2 }^2 + \kappa_0 \| \partial_x^2  u^{n+1, \ell} \|_{L^2  } ^2\notag
    \\&\leq -\int_{\T_L} (\partial_x ^2 u^{n+1,\ell} )
    \kappa'( \smoothing (u^{n,\ell}) )
    \smoothing (\partial_x u^{n,\ell}) \partial_x u^{n+1, \ell} +\int_{\T_L} H'(u^{n,\ell}) (\partial_x u^{n+1, \ell})( \partial_x^2 u^{n+1, \ell})\notag \\&\le
    C(T,\ell)\int_{\T_L} |\partial_x ^2 u^{n+1,\ell}|
    | \partial_x u^{n+1, \ell}|.\label{sep2304}
  \end{align}
  In the last step, we used \cref{sep2218} and \cref{sep2214} as well as the bound
  \begin{equation*}
    |\smoothing(\partial_xu^{n,\ell})|\le C(\ell)\|u^{n,l}\|_{L^\infty}\le C(\ell)
  \end{equation*}
  that follows from \cref{sep2214}. We see from \cref{sep2304} that
  \begin{equation*}
    \frac{\dif}{\dif t} \| \partial_x u^{n+1, \ell} \|_{L^2 }^2 +\kappa_0 \| \partial_x^2  u^{n+1, \ell} \|_{L^2  } ^2 \le C(T,\ell) \| \partial_x u^{n+1, \ell} \|_{L^2  }^2.
  \end{equation*}
  Therefore, we have
  \begin{equation}
    \label{sep2221}
    \|u^{n+1, \ell}\|_{L^\infty (0, T; H^1(\T_L) ) } + \| u^{n+1, \ell} \|_{L^2 (0, T; H^2(\T_L) ) } \le C(T, \ell).
  \end{equation}
  Moreover, it follows from \cref{eq:0811,sep2221} that
  \begin{equation*}
    \|\partial_t u^{n+1,\ell}\|_{L^2 ((0, T) \times \T_L)} \leq C(T, \ell).
  \end{equation*}
  By the Aubin--Lions lemma, as $n\to \infty$, $u^{n,\ell}$ converges along a subsequence strongly in both~$L^2(0, T; \mathcal{C}^{1, \alpha})$ and~$\mathcal{C}((0, T); \mathcal{C}^{\alpha})$ for every $\alpha \in (0, 1/2)$.
  We replace $(u^{n,\ell})_{n \geq 0}$ by this subsequence and let $u^\ell$ denote the limit.
  Then in particular, $u^{n, \ell} \rightarrow u^{\ell}$ uniformly on~$[0, T] \times \T_L$.
  Because $\kappa$ and $H'$ are continuous, we know that
  \begin{equation*}
    \kappa(\smoothing (u^{n, \ell}) ) \rightarrow \kappa ( \smoothing (u^\ell ) ) \quad \text{and} \quad H'(u^{n, \ell}) \rightarrow H'(u^\ell )
  \end{equation*}
  uniformly.
  It follows that the right side of \cref{eq:0811} converges to~$\partial_x ( \kappa (\smoothing (u^\ell)) \partial_x u^\ell ) - \partial_x (H(u^\ell) )$ weakly in $L^2 \big((0, T) \times \T_L\big)$, and hence strongly in $L^2 (0, T; H^{-1}(\T_L) )$.
  Therefore $u^\ell$ is a weak solution of \cref{sep2302}.

  Now suppose there are two solutions $u^\ell$ and $v^\ell$ of \cref{sep2302} in $\mathcal{B}$.
  Both $u^\ell$ and $v^\ell$ are sufficiently regular to serve as test functions, so we can write
  \begin{equation*}
    \begin{aligned}
      \frac 12\frac{\dif}{\dif t} \| &u^\ell - v^\ell \|_{L^2}^2 + \int_{\T_L} \kappa (\smoothing(u^\ell ) ) |\partial_x (u^\ell - v^\ell )|^2\\
      & = - \int_{\T_L} \big[\kappa (\smoothing (u^\ell ) ) - \kappa (\smoothing (v^\ell ) )\big] (\partial_x v^\ell) \partial_x (u^\ell - v^\ell ) + \int_{\T_L} \big[H(u^\ell) - H(v^\ell ) \big] \partial_x (u^\ell - v^\ell ).
    \end{aligned}
  \end{equation*}
  Using \cref{sep2218,eq:Hderivbd}, we find
  \begin{equation*}
    \frac{\dif}{\dif t} \| u^\ell - v^\ell \|_{L^2}^2 + \kappa_0 \| u^\ell - v^\ell \|_{H^1} ^2 \le C  \big(\| \partial_x v^\ell \|_{L^\infty}^2 + 1\big) \| u^\ell - v^\ell \|_{L^2}^2.
  \end{equation*}
  Since $v^\ell \in L^2 (0, T; H^2) \subset L^2 (0, T; \mathcal{C}^{1, 1/2})$, Gr\"onwall's inequality implies that $u^\ell = v^\ell$.
\end{proof}

\subsubsection{Removing the regularization}

We now construct a weak solution $u$ to \cref{sep2210} by passing to the limit $\ell\to 0$.
\begin{lem}\label{lem:rmreg}
  For each $T > 0$, \cref{sep2210} admits a weak solution in
  \begin{equation*}
    \widetilde{\mathcal{B}}\coloneqq L^2 (0, T; H^1(\T_L)) \cap L^\infty ([0, T]\times\T_L).
  \end{equation*}
\end{lem}
\begin{proof}
  Fix $T > 0$.
  We first note that \cref{sep2214} is uniform in $n$, so
  \begin{equation}
    \label{eq:Linftyul}
    \|u^\ell\|_{L^\infty([0,T]\times\T_L)}\le \|u_0\|_{L^\infty(\T_L)}
  \end{equation}
  for each $\ell > 0$.
  We cannot pass to the limit $\ell\to\infty$ in the estimate \cref{sep2221} directly, because the right side depends on $\ell$.
  We therefore develop an $H^1$ estimate.
  Multiplying \cref{sep2302} by $u^{\ell}$ and integrating by parts over $\T_L$, uniform ellipticity \cref{eq:alliptic} yields
  \begin{equation*}
    \frac12\frac{\dif}{\dif t} \| u^{\ell} \|_{L^2 (\T_L)}^2 + \kappa_0 \| \partial_x u^{\ell} \|_{L^2(\T_L)} ^2 \le 0.
  \end{equation*}
  It follows that
  \begin{equation*}
    \|u^{\ell}\|_{L^\infty (0, T; L^2 (\T_L)) } + \| u^{\ell} \|_{L^2 (0, T; H^1(\T_L) ) } \le \kappa_0^{-1/2} \|u_0\|_{L^2(\T_L)}.
  \end{equation*}
  We conclude that $(u^\ell)_{\ell > 0}$ and $(\partial_t u^\ell)_{\ell > 0}$ are uniformly bounded in $\widetilde{\cal B}$ and~$L^2 (0, T; H^{-1}(\T_L))$, respectively.
  By the Banach--Alaoglu theorem and the Aubin--Lions lemma, as $\ell \to 0$, $u^\ell$ converges along a subsequence weak-$*$ in $L^2 (0, T; H^1(\T_L))$ and strongly in $L^2(0, T; \mathcal{C}^\alpha)$ for every $\alpha \in (0, 1/2)$.
  Again, replace $(u^\ell)_{\ell > 0}$ by this subsequence and let $u$ denote the limit.

  We now show that the limit $u$ is a weak solution to \cref{sep2210}. Going back to \cref{sep2302},
  multiplying this equation by a test function~$\varphi\in \mathcal{C}^\infty([0,T)\times\T_L)$ and integrating, we obtain
  \begin{equation}
    \label{sep2308}
    \begin{aligned}
      &- \int_{\T_L} \varphi(0,x) u_0 (0,x) \,\dif x -\int_{[0,T] \times \T_L} (\partial_t \varphi)u^\ell \,\dif x \, \dif t\\ &
      =-\int_{[0,T] \times \T_L} (\partial_x \varphi )\kappa(\smoothing (u^\ell)) \partial_x u^\ell \,\dif x \, \dif t+ \int_{[0,T] \times \T_L} (\partial_x \varphi )H(u^\ell)\,\dif x \, \dif t.
    \end{aligned}
  \end{equation}
  Combining the uniform bound \cref{eq:Linftyul} with the assumptions \cref{sep2218,eq:Hderivbd}, we see that
  \begin{equation}
    \label{eq:coeff-loc-Lip}
    \begin{aligned}
      |\kappa(\smoothing(u^\ell)) - \kappa(\smoothing(u))| &\leq C(\|u_0\|_{L^\infty})|\smoothing(u^\ell) - \smoothing(u)|,\\
      |H(u^\ell) - H(u)| &\le C(\|u_0\|_{L^\infty}) |u^\ell - u|.
    \end{aligned}
  \end{equation}
  In particular, strong convergence of $u^\ell \to u$ in $L^2(0, T; \mathcal{C}^\alpha )$ implies
  \begin{equation*}
    \begin{aligned}
      & \int_{[0,T] \times \T_L} (\partial_t \varphi) u^\ell \,\dif x\,\dif t
      \rightarrow \int_{[0,T] \times \T_L} (\partial_t \varphi) u \,\dif x\,\dif t,\\
      &\int_{[0,T] \times \T_L} (\partial_x \varphi)H(u^\ell) \,\dif x\,\dif t\rightarrow \int \partial_x \varphi H(u)\,\dif x\,\dif t
    \end{aligned}
  \end{equation*}
  as $\ell \to 0$.

  For the diffusive term in \cref{sep2308}, we write
  \begin{equation}
    \label{sep2310}
    \begin{aligned}
      \int_{[0,T] \times \T_L} (\partial_x \varphi )&\big[\kappa ( \smoothing (u^\ell)) \partial_x u^\ell -   \kappa ( u) \partial_x u \big]\,\dif x \, \dif t
      = \int (\partial_x \varphi) \big[ \kappa ( \smoothing (u^\ell))   -   \kappa ( \smoothing (u) )\big] \partial_x u^\ell \,\dif x \, \dif t
      \\&+ \int (\partial_x \varphi )\big[\kappa ( \smoothing (u))   -   \kappa (u)\big] \partial_x u^\ell\,\dif x \, \dif t
      +\int( \partial_x \varphi ) \kappa (u) \partial_x (u^\ell - u)\,\dif x \, \dif t.
    \end{aligned}
  \end{equation}
  We control the first term on the right side using \cref{eq:coeff-loc-Lip}:
  \begin{equation*}
    \begin{aligned}
      \int_{[0,T] \times \T_L} |\partial_x \varphi| \left |\kappa ( \smoothing (u^\ell))- \kappa ( \smoothing (u) )\right| |\partial_x u^\ell |\,\dif x \, \dif t
      & \le C \| \smoothing (u^\ell - u ) \|_{L_{t,x}^2} \|\partial_x u^\ell \|_{L_{t,x}^2}
      \\&{\le C \| u^\ell - u\|_{L_{t,x}^2} \|\partial_x u^\ell \|_{L_{t,x}^2}}\to 0
    \end{aligned}
  \end{equation*} 
  because $u^\ell$ is bounded in $L^2(0, T; H^1(\T_L))$.
  For the second term on the right side of~\cref{sep2310}, we have
  \begin{equation*}
    \int_{[0,T] \times \T_L} \left | (\partial_x \varphi )\big[\kappa ( \smoothing (u))   -   \kappa (u)\big] \partial_x u^\ell \right | \,\dif x \, \dif t
    \leq C \|\smoothing (u) - u \|_{L_{t,x}^2} \| \partial_x u^\ell \|_{L_{t,x}^2}\to 0,
  \end{equation*}
  because $\smoothing$ is an approximation of the identity.
  Finally, the third term in \cref{sep2310} converges to $0$ by the weak-$*$ convergence in $L^2 (0, T; H^1{(\T_L)})$.
  Thus $u$ is a weak solution of \cref{sep2210}.
\end{proof}

\subsubsection{Regularity of the solution}

Next, we employ a bootstrap argument to show that $u$ is in fact a classical solution to \cref{sep2210}.
We use a De Giorgi-type result, which follows from \cite[Theorem~6.28]{L96}.
To state it, we recall the parabolic cylinders $Q_r$ introduced in \cref{eq:parabolic-cylinder}.
\begin{prop}
  \label{prop-sep2302}
  Fix $(t_0, x_0) \in \R \times \R^d$, $r > 0$, and measurable bounded functions $A$ and~$f$ on $Q_{2r}(t_0, x_0)$.
  Suppose $\lambda \leq A \leq \Lambda$ for some $0< \lambda \leq \Lambda$.
  Then there exist~$\alpha\in (0, 1)$, that depends on $\lambda$ and $\Lambda$, and $C  <+\infty$
  that depends, in addition, on $r$,  such that any weak solution $u$ of the parabolic equation
  \begin{equation*}
    \partial_t u = \partial_x(A\partial_xu) + \partial_x f \quad \text{in } Q_{2r}(t_0,x_0)
  \end{equation*}
  satisfies
  \begin{equation*}
    \|u\|_{\mathcal{C}^{\alpha}(Q_{r}(t_0, x_0))} \leq C \left(\|u\|_{L^\infty(Q_{2r}(t_0, x_0))} + \|f\|_{L^\infty(Q_{2r}(t_0, x_0))}\right).
  \end{equation*}
\end{prop}
With this, we can show that weak solutions of \cref{sep2210} are classical.
\begin{lem}
  \label{lem:weakisclassical}
  There exists $\alpha \in (0, 1)$ such that the following holds.
  For each $T > 0$, every weak solution $u$ of \cref{sep2210} in $\tilde{\mathcal{B}}$ is $\mathcal{C}_{\mathrm{loc}}^{2,\alpha}$ in space and $\mathcal{C}_{\mathrm{loc}}^{1,\alpha/2}$ in time, and thus classical.
  If $u_0 \in \mathcal{C}(\T_L)$, then $u$ is continuous.
  Moreover, there exists $C_Q \colon \R_+ \to \R_+$ such that for every parabolic cylinder ${Q_{2r}\subset (0, T) \times \T_L}$, we have
  \begin{equation}
    \label{eq:quantitative-weakisclassical}
    \|u\|_{\mathcal{C}^{\alpha}(Q_r)} \le C_Q(r) \|u\|_{L^\infty(Q_{2r})} \left(1 + \|u\|_{L^\infty(Q_{2r})}\right).
  \end{equation}
\end{lem}

\begin{proof}
  We will apply \cref{prop-sep2302} to \cref{sep2210} with $A = \kappa(u)$ and $f = -H(u)$.
  Let $(t_0, x_0)$ denote the ``center'' of the parabolic cylinders $Q_r$ and $Q_{2r}$.
  We are free to replace $f$ by~$f - f(t_0, x_0)$.
  Then \cref{eq:Hbd} and \cref{eq:Hderivbd} yield
  \begin{equation*}
    \|f\|_{L^\infty(Q_{2r})} \leq \|u\|_{L^\infty(Q_{2r})} \left(1 + \|u\|_{L^\infty(Q_{2r})}\right).
  \end{equation*}
  By a trivial modification of \cref{prop-sep2302} to differently-sized cylinders,
  \begin{equation*}
    \|u\|_{\mathcal{C}^{\alpha}(Q_{5r/3})} \leq C(r) \|u\|_{L^\infty(Q_{2r})} \left(1 + \|u\|_{L^\infty(Q_{2r})}\right),
  \end{equation*}
  which implies \cref{eq:quantitative-weakisclassical}.

  By \cref{sep2218} and \cref{eq:Hderivbd}, $H(u)$ and $\kappa(u)$ are also $\alpha$-H\"older.
  Next, let $\theta$ be a smooth cutoff such that $\theta|_{Q_{4r/3}} \equiv 1$ and $\theta|_{Q_{5r/3}^c} \equiv 0$.
  Applying Theorem~6.48 in \cite{L96} to $\tilde{u} \coloneqq \theta u$, we can check that~$u \in \mathcal{C}^{1,\alpha}(Q_{4r/3})$.
  Following \cite{L96}, we here use a parabolic notion of regularity that imparts half-regularity in time.
  That is, $u \in \mathcal{C}^{1,\alpha}(Q_{4r/3})$ means $u$ is $\mathcal{C}^{1,\alpha}$ in space and~$\mathcal{C}^{(1 + \alpha)/2}$ in time.
  Again, it follows that $H(u)$ and $\kappa(u)$ are $\mathcal{C}^{1,\alpha}$.
  Applying \cite[Theorem~6.48]{L96} with a cutoff between $Q_r$ and $Q_{4r/3}$, we see that $u$ is in fact $\mathcal{C}^{2, \alpha}$ on $Q_r$.
  That is, $u$ is $\mathcal{C}^{2,\alpha}$ in space and $\mathcal{C}^{1,\alpha/2}$ in time.
  In particular, $u$ is a classical solution of \cref{sep2210}.
  Standard boundary regularity now implies that $u$ is continuous if $u_0 \in \mathcal{C}(\T_L)$.
\end{proof}

\subsubsection{Uniqueness of the solution}

To complete the proof of \cref{prop-sep2304}, we must prove uniqueness for the periodized problem \cref{sep2210}.
To do so, we use $L^1$-contraction of the conservation law.
We recall the space~$\tilde{\mathcal{B}} = L_t^2H_x^1 \cap L_{t,x}^\infty$ from \cref{lem:rmreg}.
\begin{prop}
  \label{prop:prop0916}
  For each $T > 0$, let $u, v \in \tilde{\mathcal{B}}$ be weak solutions of \cref{eq:uPDE-unforced} on $\T_L$.
  Then
  \begin{equation*}
    \|u(t,\anon) - v(t,\anon)\|_{L^1 (\T_L)} \leq \|u(0, \anon) - v(0,\anon)\|_{L^1 (\T_L ) } \quad \text{for all } t \in [0, T].
  \end{equation*}
\end{prop}
\begin{proof}
  We first define the increasing function
  \begin{equation}
    \label{sep2410}
    \mathcal{K}(s) = \int_0^s \kappa(r) \,\dif r.
  \end{equation}
  Then, we can write \cref{eq:uPDE-unforced} as
  \begin{equation*}
    \partial_t u = \partial_x ^2 [\mathcal{K}(u)] - \partial_x [H(u)],
  \end{equation*}
  which gives
  \begin{equation}
    \label{eq:eq09162}
    \partial_t (u-v) = \partial_x ^2 [\mathcal{K}(u)- \mathcal{K}(v)] - \partial_x [H(u) - H(v)].
  \end{equation}
  
  Let $F(x) \coloneqq |x|$.
  We introduce a family of smooth convex approximations $\{F_\eps\}_{\eps>0}\in \mathcal{C}^2 (\R)$ such that
  \begin{equation}\label{sep2412}
    \lim_{\eps\rightarrow 0} \| F_\eps - F\|_{\mathcal{C}_b (\R)} = 0
  \end{equation}
  and
  \begin{equation}\label{sep2414}
    F_\eps(x)=F(x),~~\hbox{for $|x|\ge 1$.}
  \end{equation}
  As in Lemma~3.4 of \cite{DGR21}, we assume that there exists a constant $C>0$ such that for all~$\eps\in (0,1]$, the approximation $F_\eps$ satisfies the following properties for all $x \in \R$:
  \begin{equation}\label{sep2320}
    \begin{aligned}
      F_\eps (x) &\le C (|x| + \eps),                              \\
      |x F_\eps'(x) | &\le C F_\eps (x),                        \\
      |F_\eps'(x) |&\le C,                                       \\
      |x| F_\eps''(x) &\le C \one_{[-\eps, \eps]}(x).
    \end{aligned}
  \end{equation}
  Finally, we assume that $F_\eps'(w)\to \sgn(w)$ pointwise as~$\eps\to 0$, under the convention~$\sgn(0) = 0$.
  It not difficult to verify the existence of such a family $\{F_\eps\}_{\eps>0}$.

  Define $w \coloneqq \mathcal{K}(u) - \mathcal{K}(v)$ and
  \begin{equation*}
    A \coloneqq \max\big\{\|u\|_\infty, \|v\|_\infty\big\} < \infty.
  \end{equation*}
  We multiply  \cref{eq:eq09162} by $F_\eps'(w)$ and integrate in space:
  \begin{equation}
    \label{eq:eq09163}
    \int_{\T_L} \partial_t (u-v) F_\eps'(w) = - \int_{\T_L} F_\eps''(w) (\partial_x w)^2 + \int_{\T_L} F_\eps''(w) (\partial_x w) [H(u) - H(v)].
  \end{equation}
  To bound the second term in the right side above, we note that \cref{eq:Hderivbd} yields
  \begin{equation}
    \label{sep2316}
    |H(u) - H(v)| \le C(A)|u-v|.
  \end{equation}
  Here and below, we denote by $C(A)$ various constants that depend on $A$ and may vary from line to line.
  By \cref{eq:alliptic},
  \begin{equation*}
    |w| = \left|\int_{v}^{u} \kappa(s) \,\dif s\right| \geq \kappa_0 |u - v|.
  \end{equation*}
  By \cref{sep2316}, we find
  \begin{equation*}
    |H(u) - H(v)| \le C(A)|w|.
  \end{equation*}
  Now the second term on the right side of \cref{eq:eq09163} can be estimated using the last assumption in \cref{sep2320}:
  \begin{align*}
    \int_{\T_L} F_\eps'' (w) (\partial_x w) [H(u) - H(v)] &\le \frac{1}{2} \int_{\T_L} F_\eps'' (w)  (\partial_x w)^2 + C(A) \int_{\T_L}  F_\eps''(w) w^2\notag\\
                                                          &\le \frac{1}{2} \int_{\T_L} F_\eps'' (w) (\partial_x w)^2 + C(A) \int_{\T_L} |w| \one_{|w|\le\eps}\notag\\
                                                          &\le \frac{1}{2} \int_{\T_L} F_\eps'' (w) (\partial_x w)^2 + C(A,L)\eps.
  \end{align*}
  Using this estimate in \cref{eq:eq09163}, we find
  \begin{equation}
    \label{sep2321}
    \int_{\T_L} \partial_t (u-v) F_\eps'(w) + \frac{1}{2} \int_{\T_L} F_\eps''(w) (\partial_x w)^2 \le C(A,L)\eps.
  \end{equation}
  The function $\mathcal{K}(u)$ is strictly increasing, so $\sgn(w)=\sgn(u-v)$ and $F_\eps'(w) \to \sgn(u-v)$ as~$\eps\to 0$.
  Also, \cref{lem:weakisclassical} implies that $\partial_t(u-v)$ is bounded when $t > 0$.
  Taking $\eps \to 0$ in~\cref{sep2321}, the bounded convergence theorem implies
  \begin{equation}
    \label{sep2322}
    \int_{\T_L} \partial_t (u-v) \sgn(u-v) \le 0.
  \end{equation}

  Now fix $0 < t_1 \leq t_2 \leq T$.
  By \cref{lem:weakisclassical}, $u$ is continuous in spacetime.
  Noting that
  \begin{equation}
    \frac{\dif}{\dif t}\int_{\T_L} F_\eps(u-v) = \int_{\T_L} F_\eps'(u-v)\partial_t(u-v),
  \end{equation}
  we therefore have
  \begin{equation}
    \label{sep2604}
    \int_{\T_L} F_\eps(u-v)\, \Big|_{t = t_1}^{t = t_2} = \int_{[t_1, t_2] \times \T_L} F_\eps'(u-v)\partial_t(u-v).
  \end{equation}
  Again, \cref{lem:weakisclassical} and the bounded convergence theorem allow us to take $\eps \to 0$ in \cref{sep2604}:
  \begin{equation*}
    \int_{\T_L} |u(t_2, \anon)-v(t_2,\anon)| = \int_{\T_L} |u(t_1,\anon)-v(t_1,\anon)| + \int_{[t_1,t_2] \times \T_L} \partial_t(u-v)\sgn(u-v).\label{sep2606}
  \end{equation*}
  Taking $t_1 \to 0$, \cref{sep2322} and the spacetime continuity of $u$ imply the lemma.
\end{proof}
\noindent
This completes the proof of \cref{prop-sep2304}.

\subsection{The solution on the whole line}

We now send the size of the torus in the periodized problem to infinity.
Given $L>0$, let~$\chi^{[L]} \in \mathcal{C}_c^\infty(\R)$ be a bump function such that $0 \leq \chi^{[L]} \leq 1$ and
\begin{equation*}
  \begin{aligned}
    \chi^{[L]}(x) &=1 \quad \text{when } |x| \leq L/2,\\
    \chi^{[L]}(x) &=0 \quad \text{when } |x| \geq (L+1)/2.
  \end{aligned}
\end{equation*}
We assume that the family $\{\chi^{[L]}\}_{L \geq 1}(x)$ is uniformly smooth, meaning there exist constants~$C_k \in \R_+$ for $k\in\N$ such that
\begin{equation*}
  \begin{aligned}
    \|\chi^{[L]}\|_{\mathcal{C}^k(\R)} \le C_k \quad \text{for each } k \in \N
  \end{aligned}
\end{equation*}
and
\begin{equation*}
  |\partial_x\chi^{[L]}|^2 \le C_1 \chi^{[L]}.
\end{equation*}
Given $v \colon \R \to \R$, we define its $L$-periodization $v^{[L]}$ by
\begin{equation*}
  v^{[L]} (x) \coloneqq \sum_{j \in \Z} \chi^{[L]}(x - jL)v(x - jL).
\end{equation*}
In particular, if $v \in\mathcal{C}_{\p_\ell}$ then $v^{[L]} \in \mathcal{C}(\T_L)$.

Now fix $\ell \in (0, 1)$ and an initial condition $u_0\in\mathcal{C}_{\p_\ell}$.
Let $u^{[L]}$ denote the solution of \cref{sep2210} with initial condition $u_0^{[L]}$, the $L$-periodization of $u_0$.
Our first step is to prove a uniform bound on the periodized solutions in weighted spaces.
The following proposition is similar to \cite[Prop.~2.10]{DGR21}.
\begin{prop}
  \label{prop-sep2306}
  Fix $L \geq 1$, $0 < \ell < \ell' < 1$, and $A \in \R_+$ and suppose that $\|u_0^{{[L]}}\|_{\mathcal{C}_{\p_\ell}} \le A$.
  Then there exists $C \in \R_+$ depending only on $\ell$, $\ell'$, and $A$ such that
  \begin{equation}
    \label{sep2704}
    \sup_{t \in [0,1]} \| u^{[L]} (t, \anon ) \|_{\mathcal{C}_{\p_{{\ell'}}}} \le C.
  \end{equation}
\end{prop}
\begin{proof}
  As in the proof of \cite[Prop.~2.10]{DGR21}, we start by defining a time-dependent weight $a$.
  Fix~$K \ge A+1 $, to be chosen later.
  Define $\ell_1 = (\ell + \ell')/2 \in (\ell, \ell')$, $\eps = \ell' - \ell_1 = (\ell' - \ell)/2$, and
  \begin{equation*}
    a(t,x) = K^{-1} \Big(\langle x \rangle ^2 + K^{2/(1-\ell')}\Big)^{-(\ell_1 + \eps t )/2}.
  \end{equation*}
  The function $z = a u^{[L]}$ satisfies the inequality $\| z(0, \anon ) \|_{L^\infty (\R ) } \le 1$ and the differential equation
  \begin{equation}
    \label{eq:eq09102}
    \begin{aligned}
      \partial_t z = z \partial_t (\log a ) &+ \kappa(u^{[L]}) \left [ \partial_x^2 z - z \left (\partial_x ^2 (\log a) - (\partial_x (\log a ) )^2 \right ) - 2 (\partial_x z)(\partial_x (\log a ) )\right] \\
      &+ \frac{\kappa'(u^{[L]})}{a} \left ( \partial_x z - \partial_x (\log a) z \right )^2 - H'(u^{[L]}) \partial_x z + H'(u^{[L]}) \partial_x (\log a) z.
    \end{aligned}
  \end{equation}
  Next, suppose $\|z \|_{\mathcal{C}([0,1] \times \R )} > 2$.
  By \cref{prop-sep2304}, $z$ is continuous.
  Since ${\|z(0, \anon) \|_{L^\infty (\R ) } \le 1}$ and $z(t,x)$ vanishes as~$x\to\pm\infty$,
  there exists a first time $t_* \in (0, 1)$ and a position $x_* \in \R$ such that
  \begin{equation*}
    |z(t_*, x_*) |  = 2 = \max_{[0, t_*] \times \R } |z|.
  \end{equation*}
  Let us assume without loss of generality that
  \begin{equation}
    \label{sep2404}
    z(t_*,x_*)=2,
  \end{equation}
  so that
  \begin{equation}
    \label{sep2402}
    \partial_t z(t_*,x_*)\ge 0,\quad \partial_x z(t_*,x_*)=0,\quad \text{and }\quad\partial_x^2z(t_*,x_*)\le 0.
  \end{equation}
  We can easily check that
  \begin{equation}
    \label{sep2326}
    0 \le a(t,x) \le 1, \quad
    |\partial_x (\log a) | \le 1, \quad
    \left | \frac{\partial_x (\log a ) }{a } \right | \le 1, \quad
    | \partial_x ^2 \log a | \le 3,
  \end{equation}
  and
  \begin{equation}
    \label{sep2327}
    \partial_t (\log a) = -\frac{\eps}{2}\log \Big(\langle x \rangle ^2 + K^{2/(1-\ell')} \Big)\le 0.
  \end{equation}
  We control the last term in \cref{eq:eq09102} using \cref{eq:Hderivbd} and \cref{sep2326}:
  \begin{equation}
    \label{sep2325}
    |H'(u^{[L]})z||\partial_x(\log a)|\le C(1+|u^{[L]}|)|az| \frac{|\partial_x(\log a)|}{|a|}\le C(1+|z|^2).
  \end{equation}
  We evaluate \cref{eq:eq09102} at $(t_*,x_*)$ and use the bounds  \cref{eq:alliptic}, \cref{sep2218}, \cref{sep2404}--\cref{sep2326}, and~\cref{sep2325}:
  \begin{equation*}
    0 \le 2\partial_t (\log a)|_{(t_*,x_*)} + C_1,
  \end{equation*}
  with a constant $C_1$ that does not depend on $K$.
  Using \cref{sep2327}, this becomes
  \begin{equation*}
    0 \le -\eps\log \left (\langle x_* \rangle ^2 + K^{2/(1-\ell')} \right ) +C_1.
  \end{equation*}
  If we choose $K$ large enough, independent of $L$,  the right side becomes negative---a contradiction.
  We conclude that if $K$ is sufficiently large,
  then $\|z \|_{C([0,1] \times \R )} \le 2$.
  That is,
  \begin{equation*}
    |u^{[L]}(t,x) | \le 2 K \left ( \langle x \rangle^2  + K^{2/(1-\ell')} \right )^{(\ell_1 + \eps t)/2} \le 2 K  \left ( \langle x \rangle^2  + K^{2/(1-\ell')} \right )^{\ell'/2}
  \end{equation*}
  for $(t,x) \in [0,1] \times \R$.
  Hence
  \begin{equation*}
    \| u^{[L]} (t, \anon) \|_{\mathcal{C}_{\p_{\ell'}}} \le C(K, \ell') \quad \text{for all } t \in [0, 1],
  \end{equation*}
  as desired.
\end{proof}
By \cref{prop-sep2306}, the regularity bound \cref{sep2212}, and \cite[Proposition~B.2]{DGR21}, the sequence~$\{u^{[L]}\}_{L \geq 1}$ is precompact in the topology of $\mathcal{C}(0,T;\mathcal{C}_{\p_{\ell'}})$ for any $\ell' > \ell$.
Diagonalizing, there exists a subsequential limit $u$ in $\mathcal{C}(0,T;\mathcal{X}_\ell)$.
It is straightforward to check that the limit solves~\cref{eq:uPDE-unforced} on the whole line and satisfies the weighted~$L^\infty$-bound~\cref{sep2704} in \cref{prop-sep2306}.
That is,
\begin{equation}
  \label{sep2706}
  \sup_{t \in [0,1]} \| u  (t, \anon) \|_{\mathcal{C}_{\p_{{\ell'}}}(\R)} \leq C(\| u_0 \|_{\mathcal{C}_{\p_{{\ell}}}(\R)}).
\end{equation}

To complete the proof of \cref{thm-sep2702}, we must show uniqueness and continuity with respect to the initial condition.
We first prove stability in a weighted $L^1$ space on the whole line.
Given $\ell \in (0, 1)$, define
\begin{equation*}
  \zeta(x) = \e^{2^{1-\ell} - \langle x \rangle ^{1-\ell}}
\end{equation*}
The following is an analogue of Lemma~3.5 of \cite{DGR21}.
\begin{prop}
  \label{prop-sep2402}
  Let $\ell \in (0, 1)$ and $T \in \R_+.$
  Let $u, v \in L^\infty(0, T; \mathcal{C}_{\p_{\ell}}(\R))$ solve \cref{eq:uPDE-unforced} with initial data in $\mathcal{C}_{\p_\ell}(\R)$.
  Let
  \begin{equation}
    \label{sep2702}
    K \coloneqq \max\left\{\|u\|_{L^\infty(0, T; \mathcal{C}_{\p_{\ell}}(\R))},\,\|v\|_{L^\infty(0, T; \mathcal{C}_{\p_{\ell}}(\R))}\right\}.
  \end{equation}
  Then there exists $C \in \R_+$ depending on $\ell,$ $\kappa_0$, and $C_H$ from \cref{eq:Hderivbd} such that for all $t \in [0, T]$,
  \begin{equation}
    \label{sep2710}    
    \int_{\R}|u(t,x) - v(t, x)| \zeta(x) \,\dif x \le \e^{C(K + 1)t} \int_{\R} |u(0,x) - v(0, x)| \zeta(x) \,\dif x.
  \end{equation}
\end{prop}
\begin{proof}
  We follow the proof of Lemma~3.5 of \cite{DGR21}.
  Throughout, let $C$ denote a positive constant depending only on $\ell,\kappa_0, C_H$.
  We allow $C$ to change from line to line.
  
  First note that
  \begin{equation}
    \label{sep2416}
    \p_\ell |\partial_x \zeta | + |\partial_x ^2 \zeta | \le C \zeta.
  \end{equation}
  Recall the function $\mathcal{K}(u)$ defined in \cref{sep2410}.
  We define $\eta = u-v$, $w \coloneqq \mathcal{K}(u) - \mathcal{K}(v)$, and
  \begin{equation*}
    \xi = \frac{u-v}{w} \int_0^1 H' (\lambda u + (1-\lambda) v ) \,\dif\lambda = \frac{ \int_0^1 H' (\lambda u + (1-\lambda) v ) \,\dif\lambda} { \int_0 ^1 \kappa ( \lambda u + (1-\lambda) v ) \,\dif\lambda}.
  \end{equation*}
  Using \cref{eq:uPDE-unforced}, we can write an evolution equation for $\eta$ as
  \begin{equation}
    \label{eq:eq0918}
    \partial_t \eta = \partial_x ^2 w - \partial_x (\xi w).
  \end{equation}
  Note that \cref{eq:alliptic}, \cref{eq:Hderivbd}, and \cref{sep2702} imply
  \begin{equation}
    \label{sep2420}
    \|\xi\|_{\mathcal{C}_{\p_{\ell}}}\le C(1 + \|u\|_{ \mathcal{C}_{\p_{\ell}}}+\|v\|_{ \mathcal{C}_{\p_{\ell}}})\le C(K + 1).
  \end{equation}
  We approximate $F(x) = |x|$ by convex functions $F_\eps$ as in \cref{sep2412}--\cref{sep2320}.
  Multiplying~\cref{eq:eq0918}  by~$F_\eps ' (w) \zeta $ and integrating by parts in space, we find
  \begin{align*}
    \int_{\R} (\partial_t \eta) F_\eps ' (w) \zeta &= \int_{\R} \big[F_\eps ' (w) \partial_x^2 w - F_\eps' (w) \partial_x ( \xi w)\big] \zeta\\
                                                   & =  \int_{\R} \big[ \partial_x(F_\eps ' (w) \partial_x w) -F_\eps''(w)(\partial_xw)^2-  \partial_x (F_\eps' (w) \xi w) +F_\eps''(w)\xi w\partial_xw\big] \zeta\\
                                                   & = \int_{\R} \left [\partial_x ^2 (F_\eps (w)) - \partial_x (F_\eps ' (w) \xi w ) - F_\eps ''(w) (\partial_x w)^2 + F_\eps ''(w) \xi w \partial_x w  \right] \zeta\\
                                                   & = \int_{\R} [F_\eps (w) \partial_x^2 \zeta + F_\eps '(w)\xi w  \partial_x \zeta - F_\eps '' (w) ((\partial_x w)^2 - \xi w \partial_x w) \zeta].
  \end{align*}
   Young's inequality tells us that $(\partial_x w)^2 - \xi w \partial_x w\ge \frac12(\partial_x w)^2 -\frac12 \xi^2w^2\ge -\frac12 \xi^2w^2$, so since $F_\eps''\ge 0$ we obtain
  \begin{equation}
    \label{sep2418}
    \int_{\R} (\partial_t \eta) F_\eps ' (w) \zeta \le \int_{\R} \left [ F_\eps (w) \partial_x^2 \zeta + F_\eps'(w)\xi w  \partial_x \zeta + \frac{1}{2} F_\eps''(w) \xi^2 w^2 \zeta \right ].
  \end{equation}
  
  We successively bound the terms on the right side of \cref{sep2418}.
  For the first term,
  note that~\cref{eq:alliptic}, \cref{sep2320}, and the definition of $\mathcal{K}$ imply that
  \begin{equation}
    \label{eq:convex-approx-lipschitz}
    |F_\eps(w)|\le C(|w|+\eps)=C(|\mathcal{K}(u)-\mathcal{K}(v)|+\eps)\le C(|\eta| + \eps).
  \end{equation}
  Hence \cref{sep2416} yields 
  \begin{equation}
    \label{eq:boundfirstterm}
    \int_\R |F_\eps (w) \partial_x^2 \zeta| \le C \int_{\R} F_\eps (w) \zeta \le C \int_{\R} (|\eta| + \eps ) \zeta.
  \end{equation}
  To control the second term in \cref{sep2418}, we use \cref{sep2320}, \cref{sep2416}, and \cref{eq:convex-approx-lipschitz}:
  \begin{equation}
    \label{eq:boundsecondterm}
    \int_{\R} |F_\eps'(w)w \xi \partial_x \zeta|
    \le C \int_{\R}F_\eps (w) |\xi| \frac{\zeta}{\p_{\ell}}
    \le C \| \xi \|_{\mathcal{C}_{\p_{\ell}} (\R)}\int_{\R}(|\eta|+\eps) \zeta.
  \end{equation}
  Finally, the last term on the right side of \cref{sep2418} is bounded using the last inequality in~\cref{sep2320}:
  \begin{equation}
    \label{eq:boundlastterm}
    \int_{\R} |F_\eps''(w) \xi^2 w^2 \zeta| \le C\eps \| \xi \|_{\mathcal{C}_{\p_{\ell}}(\R)}^2 \|\p_\ell^2 \zeta \|_{L^1 (\R)}.
  \end{equation}
  
  Using \cref{eq:boundfirstterm}--\cref{eq:boundlastterm} in \cref{sep2418}, we have
  \begin{align*}
    \int_{\R} (\partial_t \eta) F_\eps ' (w) \zeta \le C(1 + \| \xi \|_{\mathcal{C}_{\p_{\ell}} }) \int_{\R} (|\eta| + \eps ) \zeta + C\eps \| \xi \|_{\mathcal{C}_{\p_{\ell}}}^2 \|\p_\ell^2 \zeta \|_{L^1}.
  \end{align*}
  Using \cref{sep2420} and $\zeta,\p_\ell^2\zeta \in L^1(\R)$, this becomes
  \begin{equation*}
    \int_{\R} (\partial_t \eta) F_\eps ' (w) \zeta \le C(K + 1)\int_{\R} |\eta| \zeta + C(K + 1)\eps.
  \end{equation*}
  We now take $\eps\to 0$ as in the proof of \cref{prop:prop0916}.
  Since $\sgn w =\sgn \eta$, we find
  \begin{equation}
    \label{sep2608}
    \begin{aligned}
      \frac{\dif}{\dif t}\int_{\R} |\eta|\zeta = \int_{\R} (\partial_t \eta) \sgn(\eta) \zeta \le C(K + 1) \int_{\R} |\eta|  \zeta.
    \end{aligned}
  \end{equation}
  Standard parabolic boundary estimates imply that $u$ and $v$ are continuous.
  Since $\p_\ell \ll \zeta^{-1}$ at infinity, the uniform bound \cref{sep2702} implies that the map $t \mapsto \int_{\R} |\eta(t,\anon)|\zeta$ is continuous on~$[0, T]$.
  Hence \cref{sep2710} follows from \cref{sep2608} and Gr\"onwall's inequality.
\end{proof}
Taking $u(0, \anon)=v(0, \anon)$ in the previous proposition, we obtain uniqueness for \cref{eq:uPDE-unforced}.
\begin{cor}
  Fix $m\in(0,1)$ and $T > 0$.
  Then \cref{eq:uPDE-unforced} admits a unique solution in  $\mathcal{C}(0, T; \mathcal{X}_m)$ for each initial condition in~${\mathcal X}_m$.
\end{cor}

\subsection{Continuity of the solution map}

We can now prove the main theorem of this section.
\begin{proof}[Proof of \cref{thm-sep2702}]
  Fix $m \in (0, 1)$ and $s > 0$ and consider a sequence of initial conditions~$(u_{0,n})_{n \in \N}$ in $\mathcal{X}_m$ such that
  \begin{equation}
    \label{eq:u0conv}
    u_{0,n} \to u_0 \quad \text{in } \mathcal{X}_m \text{ as } n \to \infty
  \end{equation}
  for some $u_0 \in \mathcal{X}_m$.
  Let $v_n \coloneqq \Psi_s(u_{0,n})$ and $v \coloneqq \Psi_s(u_0)$, and
  fix $\ell \in (m, 1)$.
  We will show that~$v_n \to v$ in $\mathcal{C}_{\p_{\ell}}(\R)$ as $n \to \infty$.
  Since $\ell \in (m, 1)$ is arbitrary, this will imply that $v_n \to v$ in $\mathcal{X}_m$, as desired.

  Take $m < \ell'' < \ell' < \ell$.
  By \cref{eq:u0conv}, the sequence $(u_{n,0})$ is uniformly bounded in $\mathcal{C}_{\p_{\ell''}}(\R)$.
  Hence \cref{sep2706} implies that $(v_n)$ is uniformly bounded in $\mathcal{C}_{\p_{\ell'}}(\R)$.
  By Proposition~B.1 of \cite{DGR21}, it suffices to show that $v_n \to v$ locally uniformly in $\R$.
  That is, locally uniform convergence implies that $v_n \to v$ in $\mathcal{C}_{\p_\ell}$.

  Fix compact $K, K' \subset \R$ such that $K \subset \operatorname{int} K'$.
  We know $(v_n)$ is uniformly bounded on~$K'$.
  Moreover, by the interior regularity \cref{sep2212}, the sequence $(v_n)$ is uniformly bounded in~$\mathcal{C}^\alpha(K)$.
  Hence the sequence is equicontinuous on $K$.
  On the other hand, \cref{prop-sep2402} implies that $v_n \to v$ in $L^1(K)$.
  Equicontinuity allows us to upgrade this convergence to~$L^\infty(K)$.
  Therefore $v_n \to v$ locally uniformly, and the proof is complete.
\end{proof}

\subsection{The Hamilton--Jacobi equation}
The relationship between conservation laws and Hamilton--Jacobi equations is well-known: a solution to a conservation law is the derivative of a solution to a Hamilton--Jacobi equation. In the present weighted-space setting, we have established the well-posedness theory first for stochastic conservation laws.
We  now extend the theory to the corresponding Hamilton--Jacobi equation.
\begin{prop}
  \label{prop:hamiltonjacobi}
  Let $u$ be a solution to \cref{eq:uPDE}. Fix a smooth, compactly-supported function~$\zeta$ such that $\int_\R \zeta(z)\,\dif z=1$, and
  define
  \begin{equation}
    \label{eq:hintermsofu}
    h(t,x) = \int_\R \zeta(z)\int_z^x u(t,y)\,\dif y\,\dif z + \int_0^t\!\!\int_\R \zeta(z)[\kappa(u(s,z))\partial_xu(s,z)-H(u(s,z))]\,\dif z\,\dif s.
  \end{equation}
  Then
  \begin{equation}
    \label{eq:upartialxh}
    u = \partial_xh
  \end{equation}
  and $h$ solves the Hamilton--Jacobi equation
  \begin{equation}
    \label{eq:hHJPDE}
    \partial_t h =\kappa(u(t,x))\partial_x^2h(t,x)-H(\partial_x h(t,x))+V(t,x).
  \end{equation}

\end{prop}
\begin{proof}
  The property \cref{eq:upartialxh} is clear by differentiating \cref{eq:hintermsofu} in $x$.
  Differentiating \cref{eq:hintermsofu} in time and applying \cref{eq:uPDE}, we obtain
  \begin{align*}
    \partial_t h(t,x) &= \int_\R \zeta(z)\left[\kappa(u(t,y))\partial_xu(t,y)-H(u(t,y))+V(t,y)\right]\Big|_{y=z}^{y=x}\,\dif z\\
                      &\hspace{3cm}+ \int_\R \zeta(z)[\kappa(u(t,z))\partial_xu(t,z)-H(u(t,z))]\,\dif z\\
                      &=\kappa(t,x)\partial_xu(t,x)-H(u(t,x))+V(t,x).
  \end{align*}
  In the final identity, we have used $\int_\R \zeta(z)\,\dif z=1$.
  Recalling \cref{eq:upartialxh}, we obtain \cref{eq:hHJPDE}.
\end{proof}

\section{Existence of spacetime-stationary solutions\label{sec:existence}}

In this section, we prove the existence of spacetime-stationary
solutions to~\cref{eq:uPDE}.
More precisely, we show that the set $\overline{\mathscr{P}}_{G}(\mathcal{A}_{1})$ defined in \cref{eq:PGbar}
is nonempty. We first need some estimates on the solutions,
obtained in \cref{subsec:derivbd,subsec:solnbound}.
We prove the main result of this section in \cref{subsec:existenceofinvariantmeasures}.

\subsection{Derivative bound\label{subsec:derivbd}}

We begin with an $L^{2}$ bound on the derivative of the solution.
\begin{lem}
  \label{lem:derivbd}
  Let $u$ solve \cref{eq:uPDE} with constant initial condition
  $u(0-,\anon)\equiv a \in\R$.
  Then for all $t\in[1,\infty)$ and $x \in \R$, we have
  \begin{equation}
    \label{eq:derivbd}
    \frac{1}{t}\int_{0}^{t}\E(\partial_{x}u(s,x))^{2}\,\dif s\le  
    \frac{a^2}{2 \kappa_0 t} + 
    \frac{1}{\kappa_0}\E[\partial_x\mathsf{V}_0(0)]^2.
  \end{equation}
\end{lem}
\begin{proof}
  We prove the following by induction on $k$:
  for all $k \in \Z_{\geq 0}$, $\theta \in (0, 1]$, and $x \in \R$,
  \begin{equation}
    \label{eq:energy-induction}
    \E [u((k + \theta)-, x)]^2 \leq - 2\kappa_0 \int_0^{k + \theta} \E [\partial_x u(t, x)]^2 \, \dif t + a^2 + (k + 1)\E [\partial_x\mathsf{V}_0(0)]^2.
  \end{equation}
  We begin with the base case $k = 0$.
  Because $\mathsf{V}_0$ is stationary, and hence $\partial_x\mathsf{V}_0$ is stationary and mean-zero, we have
  \begin{equation}
    \label{eq:init-moment}
    \E [u(0+, x)]^2 = \E [a + \partial_x\mathsf{V}_0(x)]^2 = a^2 + \E [\partial_x\mathsf{V}_0(0)]^2 \quad \text{for all } x \in \R.
  \end{equation}
  On the time interval $(0, \theta)$, the solution $u$ satisfies the unforced equation \cref{eq:uPDE-unforced}.
  Let $J(u)$ denote an antiderivative of $uH'(u)$.
  By the chain rule, we have
  \begin{equation*}
    \frac{1}{2}\partial_{t}\big(u^{2}\big)=u\partial_{x}\big(\kappa(u)\partial_{x}u-H(u)\big) = -\kappa(u) (\partial_x u)^2 + \partial_x \left[u \kappa(u) \partial_x u - J(u)\right].
  \end{equation*}
  We now fix $L \in \R_+$ and integrate over $(0,\theta) \times [0, L]$:
  \begin{equation*}
    \frac{1}{2}\int_0^L u^2 \,\dif x \, \Bigg|_{t = 0+}^{t = \theta-} + 
    \int_{(0, \theta) \times [0, L]} \kappa(u) (\partial_x u)^2 \,\dif x\,\dif t = \int_{0}^{\theta} \left[u\kappa(u)\partial_x u - J(u)\right] \dif t \, \Bigg|_{x = 0}^{x = L}.
  \end{equation*}
  By \cref{eq:init-moment}, the negative part of the left side is absolutely integrable.
  Thus by space-stationarity and Lemma~D.1 in \cite{DGR21}, the right side is absolutely integrable and has mean zero.
  Using~\cref{eq:init-moment},~\cref{eq:alliptic}, and space-stationarity, we find
  \begin{equation*}
    \E[u(\theta-, x)]^2 \leq -2\kappa_0 \int_0^\theta \E[\partial_x u(t, x)]^2 \, \dif t + a^2 + \E [\partial_x \mathsf{V}_0(0)]^2
  \end{equation*}
  for all $x \in \R$.
  This confirms the claim \cref{eq:energy-induction} for $k = 0$.

  Now suppose that \cref{eq:energy-induction} holds for some $k \ge 0$.
  We now show it for $k + 1$.
  Given~$\theta \in (0, 1]$,
  by an argument identical to that in the base case, we have
  \begin{equation}
    \label{eq:energy-bd}
    \begin{aligned}
      \frac{1}{2}\int_0^L u^2 \,\dif x \, \Bigg|_{t = (k + 1)+}^{t = (k + 1 + \theta)-} 
      &+ \int_{(k + 1, k + 1 + \theta) \times [0, L]} \kappa(u) (\partial_x u)^2\,\dif x\,\dif t\\
      &\hspace{3cm}= \int_{k + 1}^{k + 1 + \theta} \left[u\kappa(u)\partial_x u - J(u)\right] \dif t \, \Bigg|_{x = 0}^{x = L}.
    \end{aligned}
  \end{equation}
  Recall that $\mathsf{V}_{k + 1}$ is independent of $u((k + 1)-, \anon)$.
  Using stationarity and the inductive hypothesis \cref{eq:energy-induction} with $\theta = 1$, we therefore obtain
  \begin{align*}
    \E [u((k + 1)+, x)]^2 
    &= \E [u((k + 1)-, x)]^2 + \E[\partial_x \mathsf{V}_{k + 1}(x)]^2\\
    &\leq -2\kappa_0 \int_{(0, k + 1)} \E[\partial_x u(t, x)]^2\, \dif t + a^2 + (k + 2) \E [\partial_x\mathsf{V}_0(0)]^2 < \infty
  \end{align*}
  for all $x \in \R$.
  It follows that the negative part of the left side of \cref{eq:energy-bd} is absolutely integrable.
  Again, space-stationarity and \cite[Lemma~D.1]{DGR21} imply that the right side is absolutely integrable and has mean zero.
  Taking expectation in \cref{eq:energy-bd} and rearranging, space-stationarity and \cref{eq:alliptic} yield \cref{eq:energy-induction} for $k + 1$.
  By induction, \cref{eq:energy-induction} holds for all $k \in \Z_{\geq 0}$.
  Now fix $t \in [1, \infty)$ and $x \in \R$.
  Taking $k \coloneqq \lceil t \rceil - 1$ and $\theta \coloneqq t - k$, we can rearrange \cref{eq:energy-induction} to obtain \cref{eq:derivbd}.
\end{proof}

\subsection{Solution bound\label{subsec:solnbound}}

In this section, we fix $a \in \R$ and assume that $u$ solves \cref{eq:uPDE} with initial condition $u(0-, \anon)\equiv a$.
Our goal is to prove the following proposition.
\begin{prop}
  \label{prop:Hhbound}
  There exists a constant $C <+\infty$ depending only on $\kappa_0,\lambda,c_2$, and the law of $V$ such that for all $t\geq 1$ and $x\in\R$, we have
  \begin{equation}
    \label{eq:time-ave-mean}
    \frac{1}{t}\int_{0}^{t}\E u(s,x)\,\dif s = a
  \end{equation}
  and
  \begin{equation*}
    \frac{1}{t}\int_{0}^{t}\E H(u(s,x))\,\dif s\le C\langle a \rangle^2.
  \end{equation*}
\end{prop}
\begin{cor}
  \label{cor:Hqbound}
  With $q$ as in \assuref{H}, there exists a constant $C <\infty$, depending only on $\kappa_0,c_1,c_2,\lambda$ and the law of $V$, such that for all $t \geq 1$ and $x\in\R$, we have
  \begin{equation*}
    \frac{1}{t}\int_{0}^{t}\E|u(s,x) - a|^{q}\, \dif s \le C\braket{a}^2.
  \end{equation*}
\end{cor}
\begin{proof}
  H\"older's inequality and \cref{eq:Hbd} imply that
  \begin{equation*}
    |s - a|^q \leq  2^{q - 1} \big(|s|^q + |a|^q\big) \leq 2^{q-1}\big(c_1^{-1}(H(s)+c_1^{-1}) + |a|^q\big)
  \end{equation*}
  for all $s \in \R$.
%  Using \cref{eq:Hbd}, we can check that
%  \begin{equation*}
%    |s - a|^q \leq c_1^{-1} 2^{1 - q}\left[H(s) + c_1^{-1} + |a|^q\right].
%  \end{equation*}
  Thus \cref{prop:Hhbound} implies that
  \begin{equation*}
    \frac{1}{t} \int_0^t \E |u(s, x) - a|^q \, \dif s \leq C \braket{a}^2
  \end{equation*}
  for some $C <\infty$ depending on $\kappa_0,c_1,c_2,\lambda$, and the law of $V$.
\end{proof}
The approach to the proof of \cref{prop:Hhbound} is similar to that used in \cite{Dun20,DGR21}, based on the Cole--Hopf transform.
There is an extra step, however.
The Cole--Hopf transform cannot be applied directly, so a comparison argument is needed.

We will
first consider the case $a = 0$.
Let $g$ be the solution to the Hamilton--Jacobi equation
\begin{equation}
  \label{eq:gPDE}
  \partial_{t}g =\kappa(\partial_{x}g)\partial_{x}^2g - H(\partial_{x}g)+V, \quad g(0-, \anon) = 0,
\end{equation}
constructed in \cref{prop:hamiltonjacobi}, so that $u=\partial_{x}g$. 
Let also $\lambda,c_2$ be as in \cref{eq:Hbd} and let $h$ solve  
\begin{equation}
  \label{eq:hPDE}
  \partial_{t}h =\kappa(\partial_{x}h)[\partial_{x}^2h - \lambda(\partial_{x}h)^{2}]+V-c_2, \quad h(0-, \anon) = 0.
\end{equation}
We have the following comparison. 
\begin{lem}
  \label{lem:gandhcomp}
  For all $t>0$ and $x\in\R$, we have
  \begin{equation}
    \label{eq:hbddbyg}
    h(t,x)\le g(t,x).
  \end{equation}
\end{lem}

\begin{proof}
  By \cref{eq:alliptic} and \cref{eq:Hbd}, we know that
  \begin{equation*}
    H(u) \leq \lambda \kappa(u) u^2 + c_2.
  \end{equation*}
  Moreover, at a spatial maximum of $h-g$, we have $\partial_{x}h=\partial_{x}g=u$ and $\partial_{x}^2(h-g)\le0$.
  Subtracting \cref{eq:gPDE} from \cref{eq:hPDE}, we therefore have
  \begin{equation*}
    \partial_{t}(h-g)=\kappa(u)[\partial_{x}^2(h-g)-\lambda u^{2}]-c_2+H(u)\le0
  \end{equation*}
  at a maximum of $h - g$.
  Since $h(0, \anon) - g(0, \anon) = 0$, \cref{eq:hbddbyg} follows.
\end{proof}

This comparison is useful because $h$ admits a Cole--Hopf transformation.
\begin{lem}
  \label{lem:jan2102}
  The field
  \begin{equation}
    \label{eq:phidef}
    \phi=\e^{-\lambda h}
  \end{equation}
  solves the sequence of Cauchy problems
  \begin{align}
    & \partial_{t}\phi  =\kappa\left(-\frac{\partial_{x}\phi}{\lambda\phi}\right)\partial_{x}^2\phi+\lambda c_2\phi, &  & t\in \R_+ \setminus \N,\, x\in\R;\label{eq:phiPDE} \\
    &\phi(k+,x)        =\e^{-\lambda\mathsf{V}_{k}(x)}\phi(k-,x),                                                &  & k\in\Z_{\geq 0}, \, x\in\R;\label{eq:phijumps}           \\
    &\phi(0-)          \equiv1.\nonumber
  \end{align}
\end{lem}

\begin{proof}
  The PDE \cref{eq:phiPDE} comes from the classical Cole--Hopf transform;
  indeed, \cref{eq:phidef} implies
  \begin{equation*}
    \partial_{x}\phi=-\lambda\phi\partial_{x}h,\qquad\partial_{x}h=-\frac{\partial_{x}\phi}{\lambda\phi},\qquad\partial_{x}^2\phi=\lambda^{2}\phi(\partial_{x}h)^{2}-\lambda\phi\partial_{x}^2h
  \end{equation*}
  at non-integer times.
  Hence \cref{eq:hPDE} yields
  \begin{align*}
    \partial_{t}\phi & =-\lambda\phi\partial_{t}h=-\lambda\phi\left[\kappa(\partial_{x}h)[\partial_{x}^2h-\lambda(\partial_{x}h)^{2}]-c_2\right]=\kappa\left(-\frac{\partial_{x}\phi}{\lambda\phi}\right)\partial_{x}^2\phi+\lambda c_2\phi,
  \end{align*}
  as claimed.
  The multiplicative jump \cref{eq:phijumps} at integer times is immediate from \cref{eq:hPDE} and the definition of $V$.
\end{proof}

The next lemma gives an upper bound on the growth of a solution to
\cref{eq:phiPDE} in the deterministic case. 
\begin{lem}
  \label{lem:fkbd}
  There is a constant $C<+\infty$ depending only on $\kappa_0, \lambda$, and $c_2$, such that the following holds.
  Let $\phi$ solve the deterministic PDE \cref{eq:phiPDE} with
  \begin{equation}
    \label{eq:sub-inverse-Gaussian}
    0 \leq \phi(0, x) \leq K\e^{K x^\alpha}
  \end{equation}
  for some $K >0$, $\alpha<2$, and all $x \in \R$.
  For each $j\in\Z$, define
  \begin{equation*}
    \phi_j =\sup_{y\in[j,j+1]}\phi(0, y).
  \end{equation*}
  Then for any $t\in [0,1]$, we have
  \begin{equation}
    \label{eq:phibdintermsoff}
    \phi(t,x)\le C\sum_{j\in\Z}\phi_j\e^{-C^{-1}(x-j-1/2)^2}.
  \end{equation}
\end{lem}
\begin{proof}
  By replacing $\phi$ with $\e^{-\lambda c_2t}\phi$, we can assume that $c_2=0$.
  Let $a=\kappa_0^2/2$ and $b=\kappa_0/4$, and
  define
  \begin{equation*}
    \psi(t,x) = t^{-a} \exp\left(-\frac{bx^2}{t}\right).
  \end{equation*}
  Then
  \begin{equation*}
    \partial_t\psi(t,x)=\big(-a+bx^2t^{-1}\big)t^{-a-1} \exp\big(-bt^{-1}x^2\big)
  \end{equation*}
  and
  \begin{equation*}
    \partial_x^2\psi(t,x) = \big(-2b +4b^2x^2t^{-1}\big) t^{-a-1}\exp\big(-bt^{-1}x^2\big).
  \end{equation*}
  Thus if $\kappa\in[\kappa_0,\kappa_0^{-1}]$, $t>0$, and $x\in\R$, we have
  \begin{equation}
    \label{eq:supersolncalc}
    \begin{aligned}
      (\partial_t-\kappa\partial_x^2)\psi(t,x) &=\big[(2\kappa b-a)+b(1-4b\kappa)x^2t^{-1}\big] t^{-a-1} \exp\big(-bt^{-1}x^2\big)\\
      &=\big[\kappa_0(\kappa -\kappa_0)/2+b(1-\kappa_0\kappa)x^2t^{-1}\big]t^{-a-1} \exp\big(-bt^{-1}x^2\big)\ge0
    \end{aligned}
  \end{equation}
  since $\kappa-\kappa_0\ge 0$ and $1-\kappa_0\kappa\ge 0$.

  Let us now set
  \begin{equation*}
    B = \inf_{x\in [-1/2,1/2]}\psi(1,x) > 0,
  \end{equation*}
  and, for a given $j\in\Z$, define
  \begin{equation*}
    \psi_{j}(t, x) = \phi_j B^{-1} \psi\big(t+1,x-(j+1/2)\big)
  \end{equation*}
  and
  \begin{equation*}
    \overline\phi(t,x) = \sum_{j\in\Z}\psi_j(t,x),
  \end{equation*}
  The sum  is finite because  \cref{eq:sub-inverse-Gaussian} implies that
  \begin{equation*}
    \phi_j \leq K' \e^{K' j^\alpha}
  \end{equation*}
  for some $K' \in \R_+$.
  By \cref{eq:supersolncalc} and linearity, $\overline\phi$ is a supersolution to \cref{eq:phiPDE}.
  Moreover, by construction we have $\overline\phi(0,\anon)\ge \phi(0, \anon)$.
  By \cref{eq:sub-inverse-Gaussian}, the unique solution $\phi$ of \cref{eq:phiPDE} that grows slower than an inverse Gaussian at spatial infinity
  satisfies the comparison principle. Thus, for all $(t, x)\in (0,1] \times \R$ we have
  \begin{align*}
    \phi(t,x)\le \overline\phi(t,x)&=\frac1B\sum_{j\in\Z}\phi_j\psi\big(t+1,x-(j+1/2)\big)
                                     \le \frac1B\sum_{j\in\Z}\phi_j \e^{-b(x-j-1/2)^2/2}.
  \end{align*}
  Now \cref{eq:phibdintermsoff} follows.
\end{proof}

Next, we obtain an upper bound on the expectation of the super-solution provided
by Lemmas~\ref{lem:gandhcomp} and~\ref{lem:jan2102}. 
\begin{lem}
  \label{lem:SHE-growth}There is a constant $C<\infty$, depending only on $\kappa_0,\lambda,c_2$, and the law of $V$, such that if $\phi$ solves \cref{eq:phiPDE} and \cref{eq:phijumps} with $\phi(0-, \anon) = 1$, then for all $t\ge0$ and $x\in\R$ we have
  \begin{equation}
    \label{eq:expbd}
    \E\phi(t-,x)\le \e^{C(t + 1)}.
  \end{equation}
\end{lem}

\begin{proof}
%  For $j \in \Z$, define
%  \begin{equation*}
%    \Phi_j(t) = \sup_{x\in [j,j+1]}\phi(t,x).
%  \end{equation*}
%  and
%  \begin{equation*}
%    \overline{\mathsf{V}}_{s}(j)=\inf_{x\in[j,j+1]}\mathsf{V}_{s}(x).
%  \end{equation*}
%
%  Moreover, we can use \cref{eq:Gaussian-tail} and a union bound to control the spatial growth of $\e^{-\lambda \overline{\mathsf{V}}_s}$.
%  Define the weight $w_j \coloneqq \langle j \rangle$ for $j \in \Z$.
%  Then, we have, using \cref{eq:Gaussian-tail}, that
%  \begin{align*}
%    \P\left[\sup_{j \in \Z} w_j^{-1} \e^{-\lambda \overline{\mathsf{V}}_s(j)} > z\right] &\leq \sum_{j \in \Z} \P\left[\overline{\mathsf{V}}_s(j) < -\lambda^{-1} \log\big(z w_j)\big)\right]\\
%                                                                                         &\leq %\sum_j \exp\left[-C_1^{-1}\lambda^{-2}[(\log z)^2 + (\log w_j)^2] + C_1\right]\\
%                                                                                         &\leq %\e^{C_1} \e^{-(\log z)^2/(C_1\lambda^2)} \sum_j \exp\left[-(\log \langle j \rangle)^2/(C_1\lambda^2)\right].
%  \end{align*}
%  The last sum is finite, so
%  \begin{equation}
%\lim_{z\to\infty}   \P\left[\sup_{j \in \Z} w_j^{-1} \e^{-\lambda \overline{\mathsf{V}}_s(j)} > z\right]=0.\label{eq:sumisfinite}
%  \end{equation}
%the probability on the left side decays to zero as $z \to \infty$.
%  It follows that $\e^{-\lambda \mathsf{V}_s}$ satisfies \cref{eq:sub-inverse-Gaussian} almost surely.
The proof relies on \cref{lem:fkbd} and induction.
Recall the definition \cref{eq:Vmin} of $\mathsf{\overline{V}}_s(x)$ from \cref{assu:V}. %By \cref{eq:vbarexpmoment}, we have a constant $C_2\in\R$, depending only on the law of $\mathsf{V}$, such that
%In particular, we have
Define
  \begin{equation}
    C_1 = \E \e^{-\lambda \overline{\mathsf{V}}_s(j)},\label{eq:noisexpmoment}
  \end{equation}
  which is finite by \cref{eq:vbarexpmoment} of \cref{assu:V}.
%  for some $C_2 \in \R$ depending only on the law of $\mathsf{V}$.
For $j\in\Z$, define
  \begin{equation*}
    \Phi_j(t) = \sup_{x\in [j,j+1]}\phi(t,x).
  \end{equation*}
As our inductive hypothesis, we assume that (for a constant $C$ not depending on $t$)
\begin{equation}
\sup_{j\in\Z} \E\Phi_j(t)\le \e^{C(t+1)}\label{eq:IH-1}
\end{equation}
and that, with probability $1$, there is some $K<\infty$ and $\alpha<2$ (possibly depending on $t$) such that
\begin{equation}
 0\le \phi(t-,x)\le Ke^{Kx^\alpha}\qquad\text{for all }x\in\R.\label{eq:IH-2}
\end{equation}
This is certainly true at $t=0$; we assume it is true for $t$ and try to prove it for $t+1$.
By \cref{eq:noisexpmoment,eq:IH-1}, we have
\begin{equation}
\sup_{j\in\R} \E\Phi_j(t+)\le C_1\e^{C(t+1)},\label{eq:IH-1-again}
\end{equation}
and by the assumption (in \cref{assu:V}) that $\partial_x \mathsf{V}_t\in \mathcal{X}_0$ almost surely, we see that \cref{eq:IH-2} continues to hold (with possibly new values of $K$ and $\alpha$) when $t-$ is replaced by $t+$.
Therefore, the hypotheses of \cref{lem:fkbd} apply, so by \cref{eq:phibdintermsoff}, we see that \cref{eq:IH-2} continues to hold with $t-$ replaced by $(t+1)-$, and using \cref{eq:IH-1-again} and taking expectations in \cref{eq:phibdintermsoff}, we see that
\[
 \sup_{j\in\R} \E\Phi_j((t+1)-)\le C_2\e^{C(t+1)}\le \e^{C(t+2)}
\]
for a new constant $C_2$, with the last inequality as long as $C\ge\log C_2$. This completes the induction and thus the proof.
\end{proof}

\subsubsection*{Proof of \cref{prop:Hhbound}}

To reduce to the case $a=0$, let us define
\begin{equation*}
  \tilde{u}(t, x) \coloneqq u(t, x + 2 \lambda \kappa_0 a t) - a
  \quad \text{and} \quad
  \tilde{V}(t, x) \coloneqq V(t, x + 2 \lambda a t)
\end{equation*}
as well as $\tilde{\kappa}(s) \coloneqq \kappa(s + a)$ and
\begin{equation}
  \label{eq:new-Hamiltonian}
  \tilde{H}(s) \coloneqq H(s + a) - 2\lambda \kappa_0 a s - \lambda \kappa_0 a^2 - c_2.
\end{equation}
By space-stationarity and the independence of $V$ at different times, we know that $\tilde{V} \overset{\mathrm{law}}{=} V$.
The function $\tilde u$ satisfies 
\begin{equation}
  \label{eq:new-PDE}
  \partial_t \tilde{u} = \partial_x\big[\tilde{\kappa}(\tilde{u})\partial_x \tilde{u} - \tilde{H}(\tilde{u}) + \tilde{V}\big], \quad \tilde{u}(0-,\anon) = 0.
\end{equation}
We can use \cref{prop:hamiltonjacobi} to construct a solution $\tilde{g}$ to the Hamilton--Jacobi equation
\begin{equation*}
  \partial_t \tilde{g} = \tilde{\kappa}(\partial_x \tilde{g})\partial_x^2 \tilde{g} - \tilde{H}(\partial_x \tilde{g}) + \tilde{V}, \quad \tilde{g}(0-, \anon) = 0,
\end{equation*}
such that $\partial_x \tilde{g} = \tilde{u}$.
Note that \cref{eq:Hbd,eq:new-Hamiltonian} imply that $\tilde{H}(s) \leq \lambda \kappa_0 s^2.$
Thus, if $h$ solves~\cref{eq:hPDE} with $c_2 = 0$, \cref{lem:gandhcomp} yields $\tilde{g} \geq h$.
Drawing on \cref{lem:SHE-growth}, we find that
\begin{equation}
  \label{eq:g-moment}
  \begin{aligned}
    -\frac{C}{\lambda}(t + 1) &\overset{\cref{eq:expbd}}{\le}-\frac{1}{\lambda}\log\E\phi(t-,x)\\
    &\overset{\phantom{\cref{eq:expbd}}}{\le}\frac{1}{\lambda}\E(-\log\phi(t-,x))\overset{\cref{eq:phidef}}{=}\E h(t-,x)\overset{\cref{eq:hbddbyg}}{\le}\E\tilde{g}(t-,x),
  \end{aligned}
\end{equation}
for all $t > 0$ and $x \in \R$.
We use \cref{eq:gPDE} to write
\begin{equation*}
  \begin{aligned}
    \tilde{g}(t-,x) &= \int_0^t(\partial_{t}\tilde{g})(s,x)\,\dif s + \tilde{\mathsf{V}}_0(x)\\
    &= \int_{0}^{t}[\tilde{\kappa}(\tilde{u}(s,x))\partial_{x}\tilde{u}(s,x)-\tilde{H}(\tilde{u}(s,x))]\,\dif s + \sum_{s = 0}^{\lceil t \rceil - 1}\tilde{\mathsf{V}}_s(x).
  \end{aligned}
\end{equation*}
Here, we have  used the notation $\tilde{\mathsf{V}}_s(x) \coloneqq \mathsf{V}_s(s + 2\lambda a s)$.
This can be re-written as
\begin{equation*}
  \int_{0}^{t} \partial_{x}[\tilde{\mathcal{K}}(\tilde{u}(s, x))] \,\dif s = \tilde{g}(t-, x) + \int_{0}^{t}\tilde{H}(\tilde{u}(s,x)) \,\dif s - \sum_{s = 0}^{\lceil t \rceil - 1}\tilde{\mathsf{V}}_s(x)
\end{equation*}
for all $t > 0$ and $x \in \R$. Here, we have set
\begin{equation*}
  \tilde{\mathcal{K}}(s) \coloneqq \int_0^s \tilde{\kappa}(r) \, \dif r.
\end{equation*}
Integrating over $(0, L)$ in space for some $L > 0$, we find
\begin{equation}
  \label{eq:K-stationary}
  \int_{0}^{t} \tilde{\mathcal{K}}(\tilde{u}(s, x))\Big|_{x = 0}^{x = L} \,\dif s =
  \int_0^L \Big[\tilde{g}(t-, x) + \int_{0}^{t} \tilde{H}(\tilde{u}(s,x)) \,\dif s - \sum_{s = 0}^{\lceil t \rceil - 1}\tilde{\mathsf{V}}_s(x)\Big] \dif x.
\end{equation}
By \cref{eq:Hbd,eq:g-moment}, the negative part of the right side is absolutely integrable over $\Omega$.
Hence the same is true of the left side.
By \cite[Lemma~D.1]{DGR21} and spatial stationarity, the left side is absolutely integrable over $\Omega$ and has zero expectation.
Taking expectation in \cref{eq:K-stationary} and rearranging, spatial stationarity allows us to remove the spatial integral:
\begin{equation*}
  \E \int_{0}^{t} \tilde{H}(\tilde{u}(s,x)) \,\dif s = -\E \tilde{g}(t-, x) + \sum_{s = 0}^{\lceil t \rceil - 1} \E \tilde{\mathsf{V}}_s(x),
\end{equation*}
for all $t > 0$ and $x \in \R$.
Note that \cref{eq:Hbd} and Fubini--Tonelli allow us to exchange the expectation and integral on the left side.
Using \cref{eq:g-moment} and the stationarity of the family $\mathsf{V}$, we find
\begin{equation*}
  \int_{0}^{t} \E \tilde{H}(\tilde{u}(s,x)) \,\dif s \leq \frac{C}{\lambda}(t + 1) + \lceil t \rceil \E \mathsf{V}_0(0).
\end{equation*}
Therefore, we have 
\begin{equation}
  \label{eq:new-Hamiltonian-bd}
  \frac{1}{t}\int_{0}^{t} \E \tilde{H}(\tilde{u}) \leq C
\end{equation}
for all $t \geq 1$ and some constant $C<+\infty$  depending only on $\kappa_0,\lambda, c_1, c_2$, and the law of $V$.

To verify \cref{eq:time-ave-mean}, we integrate \cref{eq:new-PDE} in spacetime.
Given $L > 0$ and $t > 0$, we have
\begin{equation}
  \label{eq:new-mean-prelim}
  \int_{[0, L]} \tilde{u}(t-, x) \, \dif x = \left[\int_0^t (\tilde{\kappa}(\tilde{u}) \partial_x \tilde{u} - \tilde{H}(\tilde{u})) \, \dif s + \sum_{s = 0}^{\lceil t \rceil - 1} \tilde{\mathsf{V}}_s\right] \, \Bigg|_{x = 0}^{x = L}.
\end{equation}
Now, \cref{eq:Hbd} and \cref{eq:new-Hamiltonian} imply that $\tilde{H}(s) \gg |s|$ when $|s| \gg 1$.
Thus \cref{eq:new-Hamiltonian-bd} implies that $\tilde{u}(t-, x)$ has a first moment for all $x \in \R$ and almost every $t > 0$.
Thus for almost every $t > 0$, we can apply \cite[Lemma~D.1]{DGR21} to \cref{eq:new-mean-prelim} to conclude that
\begin{equation}
  \label{eq:new-mean}
  \E \tilde{u}(t-, x) = 0.
\end{equation}
This implies \cref{eq:time-ave-mean}.

We can now control $H(u)$.
Combining \cref{eq:new-Hamiltonian}, \cref{eq:new-Hamiltonian-bd}, and \cref{eq:new-mean} and using space-stationarity, we find
\begin{equation*}
  \frac{1}{t} \int_0^t \E H(u)\, \dif s = \frac{1}{t} \int_0^t \E \left[\tilde{H}(\tilde{u}) + 2 \lambda \kappa_0 a \tilde{u} + \lambda \kappa_0 a^2 + c_2\right] \dif s \leq C \braket{a}^2
\end{equation*}
for all $t \geq 1$. This completes the proof of Proposition~\ref{prop:Hhbound}.
\hfill$\Box$

\subsection{Existence of spacetime-stationary solutions\label{subsec:existenceofinvariantmeasures}}

We are now ready to prove \cref{thm:existence}.
\begin{proof}[Proof of \cref{thm:existence}.]
  Fix $a \in \R$ and let $u$ solve \cref{eq:uPDE} with initial condition $u(0-, \anon) \equiv a$.
  Let
  \begin{equation*}
    \mu_{t}=\Law(u(t))\otimes\delta_{\overline{t}},
  \end{equation*}
  where $\delta_{\overline{t}}$ is a delta mass at $\overline{t}\in\R/\Z$.
  Then $\mu_{t}$ is a probability measure on $\mathcal{A}_{1}$.
  Given $t \geq 1$, define
  \begin{equation*}
    \overline{\mu}_{t}=\frac{1}{t}\int_{0}^{t}\mu_{s}\,\dif s.
  \end{equation*}

  We claim that $(\overline{\mu}_{t})_{t\ge 1}$ is tight with respect to the topology of $\mathcal{X}_{\frac{2}{2 + q}} \times (\R/\Z)$.
  To see this, take $(v,\theta)\sim\overline{\mu}_{t}$ for fixed $t \geq 1$.
  By \cref{lem:derivbd}, \cref{prop:Hhbound},
  and \cref{cor:Hqbound}, there is a constant $C <+\infty$ depending on $\kappa_0,c_1,c_2,\lambda$, and the law of $V$ but  independent of $t$
  such that
  \begin{equation}
    \label{eq:v-bound}
    \E H(v(x)), \, \E |v(x) - a|^{q} \leq C \braket{a}^2 \quad \text{and} \quad \E|\partial_{x}v(x)|^{2} \le C \quad \text{for all } x \in \R.
  \end{equation}
  In the remainder of the proof, we allow $C$ to change from line to line provided it continues to depend only on $\kappa_0,c_1,c_2,\lambda$, and the law of $V$.
  Fix $\alpha\ge1$ and $\beta \in (0, 1)$ to be chosen later.
  Given $j\in\Z$,
  let $c_{j}=\sgn(j)\langle j\rangle^{\alpha}$. We have
  \begin{equation*}
    \|v\|_{\mathcal{C}_{\p_{\beta}}}=\adjustlimits\sup_{j\in\Z}\sup_{x\in[c_{j},c_{j+1}]}\frac{|v(x)|}{\langle x\rangle^{\beta}}\le \sup_{j\in\Z}\frac{|v(c_{j})|}{\langle c_{j}\rangle^{\beta}}+\sup_{j\in\Z}\frac{\int_{c_{j}}^{c_{j+1}}|\partial_{x}v(x)|\,\dif x}{\langle c_{j}\rangle^{\beta}}.
  \end{equation*}
  Note that \cref{eq:v-bound} yields
  \begin{equation*}
    \E\Big(\sup_{j\in\Z}\frac{|v(c_{j})|}{\langle c_{j}\rangle^{\beta}}\Big)^{q}\le\sum_{j\in\Z}\frac{\E|v(c_{j})|^{q}}{\langle c_{j}\rangle^{\beta q}}\le C\braket{a}^2\sum_{j\in\Z}\langle j\rangle^{-\alpha\beta q},
  \end{equation*}
  which is finite provided
  \begin{equation}
    \label{jan2402}
    \beta>\frac{1}{\alpha q}.
  \end{equation}
  Similarly, \cref{eq:v-bound} implies
  \begin{align*}
    \E\Big(\sup_{j\in\Z}\frac{\int_{c_{j}}^{c_{j+1}}|\partial_{x}v(x)|\,\dif x}{\langle c_{j}\rangle^{\beta}}\Big)^{2} & \le
                                                                                                                         \sum_{j\in\Z}\frac{\E\left(\int_{c_{j}}^{c_{j+1}}|\partial_{x}v(x)|\,\dif x\right)^{2}}{\langle c_{j}\rangle^{2\beta}}     \\
                                                                                                                       & \le\sum_{j\in\Z}\frac{(c_{j+1}-c_{j})\int_{c_{j}}^{c_{j+1}}\E|\partial_{x}v(x)|^{2}\,\dif x}{\langle c_{j}\rangle^{2\beta}} 
                                                                                                                         \le C\sum_{j\in\Z}\frac{(c_{j+1}-c_{j})^{2}}{\langle c_{j}\rangle^{2\beta}}  \\ &                                                      
                                                                                                                                                                                                           \le C\sum_{j\in\Z}\langle j\rangle^{2[\alpha-1-\alpha\beta]},
  \end{align*}
  which is finite when $2[\alpha-1-\alpha\beta]<-1$,
  that is,
  \begin{equation}
    \label{jan2404}
    \beta>1-{1}/{(2\alpha)}.
  \end{equation}
  Taking $\alpha= {1}/{2}+{1}/{q}$,
  we see that if $\beta>{2}/{(2+q)}$, then both \cref{jan2402} and \cref{jan2404} hold, so that
  \begin{equation}
    \label{eq:Cpbetabound}
    \E\|v\|_{\mathcal{C}_{\p_{\beta}}}^{q}\le C\braket{a}^2.
  \end{equation}
  For instance, we can take $\beta = 5/6$.

  The next step is to control the H\"older regularity of $v$.
  Given $\gamma\in(0,1/2]$, we have
  \begin{align*}
    \E\Bigg(\sup_{\substack{|x-y|\le1 \\ x \neq y}}\frac{|v(x)-v(y)|}{\langle x\rangle^{\beta}|x-y|^{\gamma}}\Bigg)^{2} & \le\sum_{j\in\Z}\E\Bigg(\sup_{\substack{x, y\in[j-1,j+1]\\x \neq y}}\frac{|v(x)-v(y)|^{2}}{\langle x\rangle^{2\beta}|x-y|^{2\gamma}}\Bigg) \\
                                                                                                                        & \le 2 \sum_{j\in\Z}\langle j\rangle^{-2\beta}\int_{j-1}^{j+1}\E|\partial_{x}v(x)|^{2}\,\dif x
                                                                                                                          \le C\sum_{j\in\Z}\langle j\rangle^{-2\beta}.
  \end{align*}
  The last sum is finite since $\beta>1/2$.
  In light of \cref{eq:Cpbetabound}, and since $1<q\le 2$, we see that
  \begin{equation}
    \label{eq:Cgammapbbd}
    \E\|v\|_{\mathcal{C}_{\p_{\beta}}^{\gamma}}^{q}\le C\braket{a}^2.
  \end{equation}

  As $\beta > {2}/{(2 + q)}$ and $\gamma > 0$, Proposition~B.2 of \cite{DGR21} ensures that the embedding
  \begin{equation*}
    \mathcal{C}_{\p_{\beta}}^{\gamma}\hookrightarrow\mathcal{X}_{\frac{2}{2+q}} = \mathcal{X}
  \end{equation*}
  is compact.
  Therefore \cref{eq:Cgammapbbd} and the compactness of $\R/\Z$ imply that the sequence $(\overline{\mu}_{t})_{t\ge0}$ of measures on $\mathcal{A}_1 = \mathcal{X} \times (\R/\Z)$ is tight.
  By Prokhorov's theorem, there exists a weak subsequential limit $\overline{\mu}$ of $(\overline{\mu}_{t})_{t\ge 1}$ as $t \to \infty$.
  Using the Feller property from \cref{cor:feller}, a standard Krylov--Bogoliubov argument shows that $\overline{\mu}$ is invariant for the semigroup $(\mathcal{P}_{t})_{t\ge0}$.
  Also, $\overline{\mu}$ is certainly invariant under spatial translations, so $\overline{\mu}\in\overline{\mathscr{P}}_{\R}(\mathcal{A}_{1})$.
  If $(\overline{v}, \overline{\theta}) \sim \overline{\mu}$, then, by~\cref{eq:v-bound}, we have
  \begin{equation}
    \label{eq:limit-moment}
    \E H(v), \, \E |\overline{v} - a|^q \leq C \braket{a}^2 \quad \text{and} \quad \E |\partial_x v|^2 \leq C,
  \end{equation}
  where we implicitly evaluate $v$ at some $x \in \R$.
  Moreover, the uniform integrability implicit in \cref{eq:v-bound} yields
  \begin{equation}
    \label{eq:limit-mean}
    \E \overline{v} = a.
  \end{equation}

  We now write $\overline{\mu}$ as a convex combination of extremal measures.
  Fix $G = \R$ or $L \Z$ for some~$L > 0$ and let $X \sim \Uniform(\R/G)$.
  Because $\overline{\mu}$ is $\R$-invariant, it is also $G$-invariant.
  In~\cite[Theorem~4.4]{Varadarajan},  it is shown  that $\bar{\mu}$ corresponds to a probability measure $m(\dif \mu)$ supported 
  on~$\overline{\mathscr{P}}_{G}^{\mathrm{e}}(\mathcal{A}_{1})$ such that
  \begin{equation}
    \label{eq:ergodic-decomp}
    \overline{\mu}(A) = \int_{\overline{\mathscr{P}}_{G}^{\mathrm{e}}(\mathcal{A}_{1})} \mu(A) \, m(\dif \mu)
  \end{equation}
  for each Borel set $A \subset \mathcal{A}_1$.
  Strictly speaking, \cite{Varadarajan} only treats deterministic dynamical systems.
  However, as noted in \cref{rem:ergodic}, we can convert our Markov semigroup to a deterministic dynamical system following \cite[Section 4]{H08}.
  
  Now if $(v[\mu], \theta[\mu]) \sim \mu$ for each $\mu \in \overline{\mathscr{P}}_{G}^{\mathrm{e}}(\mathcal{A}_{1})$, then \cref{eq:limit-moment,eq:ergodic-decomp} imply that $m$ is supported on measures $\mu$ such that
  \begin{equation}
    \label{eq:finite-moment}
    \E H(v[\mu](X)), \, \E |v[\mu](X) - a|^q, \, \E |\partial_x v[\mu](X)|^2 < \infty.
  \end{equation}
  In particular, $m$ is supported on measures with well-defined first moments.
  By the H\"older inequality,~\cref{eq:limit-moment,eq:ergodic-decomp} yield
  \begin{equation}
    \label{eq:m-moment-verbose}
    \int_{\overline{\mathscr{P}}_{G}^{\mathrm{e}}(\mathcal{A}_{1})} |\E v[\mu](X) - a|^q \, m(\dif \mu) \leq \int_{\overline{\mathscr{P}}_{G}^{\mathrm{e}}(\mathcal{A}_{1})} \E |v[\mu](X) - a|^q \, m(\dif \mu) \leq C \braket{a}^2.
  \end{equation}
  Also, we can write \cref{eq:limit-mean} as
  \begin{equation}
    \label{eq:m-mean-verbose}
    \int_{\overline{\mathscr{P}}_{G}^{\mathrm{e}}(\mathcal{A}_{1})} \E v[\mu](X) \, m(\dif \mu) = a.
  \end{equation}
  If $\mu \sim m$, let $\xi \coloneqq \E[v[\mu](X) \mid \mu] - a$.
  Then \cref{eq:m-moment-verbose,eq:m-mean-verbose} become 
  \begin{equation}
    \label{eq:m-moment}
    \E |\xi|^q \leq C \braket{a}^2
  \end{equation}
  and
  \begin{equation}
    \label{eq:m-mean}
    \E \, \xi = 0.
  \end{equation}
  Let $\mathsf{m} \coloneqq \Law \xi$, and suppose
  \begin{equation*}
    \supp \mathsf{m} \cap [-\braket{a}/2, A] = \emptyset
  \end{equation*}
  for some $A \geq \braket{a}/2$.
  Then we can use \cref{eq:m-mean} to write
  \begin{equation*}
    0 = \int_{\R} \xi \, \mathsf{m}(\dif x) = -\int_{-\infty}^{-\braket{a}/2} |\xi| \, \mathsf{m}(\dif x) + \int_A^\infty |\xi| \, \mathsf{m}(\dif x).
  \end{equation*}
  Therefore, we have
  \begin{equation*}
    \int_A^\infty |\xi| \, \mathsf{m}(\dif x) = \frac{1}{2} \E |\xi| \geq \frac{\braket{a}}{2}.
  \end{equation*}
  Using \cref{eq:m-moment}, this implies:
  \begin{equation*}
    \frac{\braket{a}}{2} \leq \int_A^\infty \xi \, \mathsf{m}(\dif x) \leq A^{1- q} \int_A^\infty \xi^q \, \mathsf{m}(\dif x) \leq A^{1- q} \E |\xi|^q \leq C A^{1-q} \braket{a}^2.
  \end{equation*}
  Rearranging, we find
  \begin{equation*}
    A \leq \frac{1}{2}\bar{C} \braket{a}^{\frac{1}{q - 1}},
  \end{equation*}
  with some $\bar{C} $ depending only on $\kappa_0,\lambda,c_1,c_2$, and $\Law V$.
  It follows that
  \begin{equation*}
    \mathsf{m}\Big(\Big[-\frac{\braket{a}}{2}, \bar{C} \braket{a}^{1/(q - 1)}\Big]\Big) > 0.
  \end{equation*}
  In light of the definition of $\xi$ and \cref{eq:finite-moment}, there exists $\mu_{a,G} \in \overline{\mathscr{P}}_{G}^{\mathrm{e}}(\mathcal{A}_{1})$ such that
  \begin{equation*}
    \E H(v[\mu_{a,G}](X)), \, \E |\partial_x v[\mu_{a,G}](X)|^2 < \infty
  \end{equation*}
  and
  \begin{equation*}
    -\braket{a}/2 \leq \E v[\mu_{a,G}](X) - a \leq \bar{C}\braket{a}^{\frac{1}{q - 1}}.
  \end{equation*}
  This completes the proof of \cref{thm:existence}.
\end{proof}

\section{Stochastic ordering of the invariant measures}
\label{sec:uniqueness}

In this section we prove \cref{thm:stochordering}. First, we show that
a coupling satisfying \cref{eq:isacoupling} exists.
\begin{prop}
  \label{prop:cancouple}
  Suppose that $\mu_{i} \in \overline{\mathscr{P}}_G(\mathcal{A}_{N_i})$ with $i \in \{1, 2\}$ satisfy the hypotheses of \cref{thm:stochordering}.
  Then there exists $\mu\in\overline{\mathscr{P}}_{G}(\mathcal{A}_{N_{1}+N_{2}})$ satisfying \cref{eq:isacoupling}.
\end{prop}
\begin{rem}
  Note that in this part of the proof of \cref{thm:stochordering}, we do not assume that $\mu_i$ is extremal.
\end{rem}
\begin{proof}
  Because each $\mu_i$ is invariant under the semigroup $\{\mathcal{P}_t\}_{t \geq 0}$, their marginals on $\R/\Z$ are uniform for each $i \in \{1, 2\}$,
  and are, therefore, identical.
  It follows that there exists a coupling $\mu_0 \in\mathscr{P}_{G}(\mathcal{A}_{N_{1}+N_{2}})$ such that if
  \begin{equation*}
    \big((v_{1;1},\ldots,v_{1;N_{1}},v_{2;1},\ldots,v_{2;N_{2}}),\theta\big) \sim \mu_{0},
  \end{equation*}
  then
  \begin{equation*}
    \Law\big((v_{i;1},\ldots,v_{i;N_{i}}),\theta\big)=\mu_{i} \quad \text{for each } i \in \{1, 2\}.
  \end{equation*}
  Deploying the Krylov--Bogoliubov-type argument in the proof of \cite[Proposition 4.3]{DGR21}, there is a sequence of times $T_k\nearrow\infty$ such that the limit
  \begin{equation*}
    \mu=\lim_{k\to\infty}\frac{1}{T_k}\int_{0}^{T_k}\mathcal{P}_{t}^{*}\mu_{0}\,\dif t
  \end{equation*}
  exists and is an element of $\overline{\mathscr{P}}_{G}(\mathcal{A}_{N_{1}+N_{2}})$.
  Moreover, the invariance of $\mu_i$ implies that $\mu$ satisfies \cref{eq:isacoupling}.
\end{proof}
Next, we show that the components of a time-invariant solution are ordered.
\begin{prop}
  \label{prop:ordering}
  Let the group $G$ be $\R$ or $L\Z$ for some ${L>0}$ and let $X\sim\Uniform(\R/G)$.
  Suppose that $\mu\in\overline{\mathscr{P}}_{G}(\mathcal{A}_{N})$ and $(v_{1},\ldots,v_{N},\theta)\sim\mu$ satisfy
  \begin{equation}
    \label{eq:momentcondition}
    \E H(v_{i}(X)),\,\E(\partial_{x}v_{i}(X))^{2}<\infty \quad \text{for all } i \in \{1,\ldots,N\}.
  \end{equation}
  Then for each $i,j\in\{1,\ldots,N\}$, $\sgn\big(v_i(x)-v_j(x)\big)$ is almost surely a random constant independent of $x \in \R$.
\end{prop}

\begin{proof}
  It suffices to consider the case $N=2$. The statement and proof are similar to those of~\cite[Proposition~3.9]{DGR21}.
  Let $(v_{1},v_{2}, \theta)\sim\mu$ be independent of the noise $V$.
  For each $i \in \{1, 2\}$, let $u_{i}$ solve \cref{eq:uPDE} with initial condition $u_{i}(\theta-)=v_{i}$ at time $\theta$.
  Since $\theta$ is uniformly distributed in $\R/\Z$, we can restrict to the full-measure event $\theta \neq 0$.

  Let $F(x)=|x|$ and let $\{F_{\eps}\}_{\eps\in(0,1]}$ be a family of functions as in~\cite[Lemma 3.4]{DGR21}:
  each~$F_{\eps}$ is convex and there is a constant $C$ so that
  \begin{equation}
    \label{eq:Fppbound}
    F_{\eps}(\xi) \le C(|\xi|+\eps), \quad |\xi F_{\eps}'(\xi)| \le CF_{\eps}(\xi), \quad |F_{\eps}'(\xi)| \le C, \quad |\xi|F_{\eps}''(\xi) \le C\one_{[-\eps,\eps]}(\xi).
  \end{equation}
  Moreover, we can assume that $F_\eps$ is independent of $\eps$ outside $[-1, 1]$.
  By \cref{eq:uPDE} and the chain rule, the difference $\eta=u_{1}-u_{2}$ satisfies
  \begin{equation}
    \label{eq:chainruleFepseta}
    \partial_{t} F_{\eps}(\eta)=F'_{\eps}(\eta)\partial_{x}\left[\kappa(u_{1})\partial_{x}u_{1}-\kappa(u_{2})\partial_{x}u_{2}-H(u_{1})+H(u_{2})\right]
  \end{equation}
  at non-integer times.
  
  Let $\hL = L$ if $G = L\Z$ for $L > 0$ and $\hL = 1$ if $G = \R$, and
  integrate \cref{eq:chainruleFepseta} over~$(\theta, 1) \times [0, \hL]$.
  Integrating by parts in space, we find:
  \begin{equation}
    \label{eq:diff-int}
    \begin{aligned}
      \int_0^{\hL} F_\eps(\eta) \, \dif x \, &\Big|_{t = \theta}^{t = 1-} = \int_\theta^1 F_\eps(\eta)\left[\kappa(u_{1})\partial_{x}u_{1}-\kappa(u_{2})\partial_{x}u_{2}-H(u_{1})+H(u_{2})\right] \dif t \,\Big|_{x = 0}^{x = \hL}\\
      &- \int_\theta^1\!\!\int_0^{\hL}
      F''_{\eps}(\eta)\partial_{x}\eta\left[\kappa(u_{1})\partial_{x}u_{1}-\kappa(u_{2})\partial_{x}u_{2}-H(u_{1})+H(u_{2})\right]\,\dif x\,\dif t.
    \end{aligned}
  \end{equation}
  Next, we write
  \begin{equation}
    \label{eq:diff-grouped}
    \kappa(u_{1})\partial_{x}u_{1}-\kappa(u_{2})\partial_{x}u_{2} = \kappa(u_1) \partial_x \eta +[\kappa(u_1) - \kappa(u_2)] \partial_xu_2.
  \end{equation}
  Recall that \assuref{kappa} states that $\kappa$ is uniformly $\alpha_\kappa$-H\"older regular for some $\alpha_\kappa \in (1/2, 1)$.
  Using this regularity, \cref{eq:Fppbound}, and Young's inequality, we find
  \begin{equation}
    \label{eq:kappa-diff}
    \begin{aligned}
      \big|F''_{\eps}(\eta)\partial_{x}\eta [\kappa(u_1) - \kappa(u_2)] \partial_xu_2\big| &\leq \|\kappa\|_{\mathcal{C}^{\alpha_\kappa}} F''_{\eps}(\eta) |\partial_x \eta| |\eta|^{\alpha_\kappa} |\partial_x u_2|\\
      &\leq \frac{\kappa_0}{4} F''_{\eps}(\eta) |\partial_x \eta|^2 + \frac{\|\kappa\|_{\mathcal{C}^{\alpha_\kappa}}^2}{\kappa_0} F''_{\eps}(\eta) |\eta|^{2\alpha_\kappa} |\partial_x u_2|^2\\
      &\leq \frac{\kappa_0}{4} F''_{\eps}(\eta) |\partial_x \eta|^2 + \|\kappa\|_{\mathcal{C}^{\alpha_\kappa}}^2 \kappa_0^{-1} \eps^{2\alpha_\kappa - 1} |\partial_x u_2|^2.
    \end{aligned}
  \end{equation}
  Similarly, \cref{eq:Hderivbd}, \cref{eq:Fppbound}, and Young's inequality imply
  \begin{equation}
    \label{eq:Hamiltonian-diff}
    \begin{aligned}
      \big|F''_{\eps}(\eta)\partial_{x}\eta [H(u_{1}) - H(u_{2})]\big| &\leq C F''_{\eps}(\eta) |\partial_x\eta| |\eta| (|u_1| + 1)^{q/2}\\
      &\leq \frac{\kappa_0}{4} F''_{\eps}(\eta) |\partial_x\eta|^2 + C^2 \kappa_0^{-1} \eps (|u_1| + 1)^{q}.
    \end{aligned}
  \end{equation}
  Combining \cref{eq:diff-grouped}, \cref{eq:kappa-diff}, and \cref{eq:Hamiltonian-diff} and using \cref{eq:alliptic} we obtain
  \begin{equation}
    \label{eq:integrand-bound}
    \begin{aligned}
      -F''_{\eps}(\eta)\partial_{x}\eta\big[\kappa(u_{1})&\partial_{x}u_{1}-\kappa(u_{2})\partial_{x}u_{2}-H(u_{1})+H(u_{2})\big]\\
      &\leq -\frac{\kappa_0}{2} F''_{\eps}(\eta) |\partial_x \eta|^2 + C \eps^{2\alpha_\kappa - 1} |\partial_x u_2|^2 + C \eps (|u_1| + 1)^{q},
    \end{aligned}
  \end{equation}
  where we allow $C$ to change from line to line.
  Now \cref{eq:momentcondition} and \cref{eq:Hbd} imply that
  \begin{equation}
    \label{eq:integral-moments}
    \E \int_\theta^1\!\!\int_0^{\hL}
    \big(|\partial_x u_2|^2 + (|u_1| + 1)^{q}\big)\,\dif x\,\dif t < \infty.
  \end{equation}
  Thus, we have
  \begin{equation}
    \label{eq:one-sided-moment}
    \begin{aligned}
      \E\Big(-\int_\theta^1\!\!\int_0^{\hL}
      F''_{\eps}(\eta)\partial_{x}\eta&
      \big[\kappa(u_{1})\partial_{x}u_{1}-\kappa(u_{2})\partial_{x}u_{2}-H(u_{1})+H(u_{2})\big] \,\dif x\,\dif t\Big)_+\\
      &\leq C \eps^{2 \alpha_\kappa - 1} \E \int_\theta^1\!\!\int_0^{\hL}
      \left[|\partial_x u_2|^2 + (|u_1| + 1)^{q}\right] \dif x\,\dif t < \infty.
    \end{aligned}
  \end{equation}
  On the other hand, $G$-invariance and time stationarity imply that
  \begin{equation}
    \label{eq:two-differences}
    \int_0^{\hL} F_\eps(\eta) \, \dif x \, \Big|_{t = \theta}^{t = 1-} - \int_\theta^1 F_\eps(\eta)\left[\kappa(u_{1})\partial_{x}u_{1}-\kappa(u_{2})\partial_{x}u_{2}-H(u_{1})+H(u_{2})\right] \dif t \,\Big|_{x = 0}^{x = \hL}
  \end{equation}
  can be written as the difference of two identically distributed random variables.
  Therefore,~\cref{eq:diff-int},~\cref{eq:one-sided-moment}, and \cite[Lemma~D.1]{DGR21} imply that the expression in \cref{eq:two-differences} is absolutely integrable and has zero expectation.
  Taking expectation in \cref{eq:diff-int}, \cref{eq:integrand-bound}, and \cref{eq:integral-moments} yields
  \begin{equation}
    \label{eq:eps-upper}
    \E \int_\theta^1\!\!\int_0^{\hL}
    F''_{\eps}(\eta) |\partial_x \eta|^2 \,\dif x\,\dif t\leq C \eps^{2 \alpha_\kappa - 1}.
  \end{equation}
  
  We now take $\eps \to 0$.
  Let $\zeta \in \mathcal{C}_c^\infty(\R)$ be nonnegative with
  \begin{equation*}
    \int_{\R} \zeta \dif x = 1,
  \end{equation*}
  and let $\tilde{\zeta} \coloneqq \zeta \ast \one_{[0, \hL]}$.
  Mimicking the proof of \cite[Proposition~3.7]{DGR21}, we can use \cref{eq:eps-upper}, $G$-invariance, and the coarea formula to show that
  \begin{equation*}
    \int_\theta^1 \sum_{y \in \eta(t, \anon)^{-1}(0)} |\partial_x \eta(t, y)| \tilde{\zeta}(y) \, \dif t = \frac{1}{2}\lim_{\eps \to 0} \E \int_\theta^1 \!\!\int_0^{\hL} F''_{\eps}(\eta(t, x)) |\partial_x \eta(t, x)|^2 \,\dif x \,\dif t = 0
  \end{equation*}
  almost surely.
  If we translate $\zeta$ (and thus $\tilde\zeta$) along $\R$, we see that with probability 1, we have 
  $\partial_x \eta(t, \anon) = 0$ wherever $\eta(t, \anon) = 0$ for almost every $t \in [\theta, 1)$.
  \cref{prop-sep2304} and~\cite[Theorem~4.8]{L96} imply that $\partial_x \eta$ is continuous in spacetime.
  Following the proof of~\cite[Lemma~3.10]{DGR21}, we can show that with probability $1,$ we have 
  $\partial_x \eta = 0$ wherever~$\eta = 0$ in~$(\theta, 1) \times \R$.
  The proof of \cite[Proposition~3.9]{DGR21} shows that this contradicts the parabolic Hopf lemma unless $\sgn\eta$ is a random constant independent of $(t, x) \in (\theta, 1) \times \R$.
  The comparison principle implies that then $\sgn\eta$ is almost surely a random constant independent of~$(t, x) \in (\theta, \infty) \times \R$.
  By time stationarity, we also have $\eta(\theta + 1, \anon) \overset{\text{law}}{=} \eta(\theta, \anon)$. 
  The proposition follows.
\end{proof}
Next, we apply this proposition to extremal solutions.
\begin{cor}
  \label{cor:extremal-ordering}
  Let $\mu$ and $(\mathbf{v}, \theta) \sim \mu$ satisfy the hypotheses of \cref{prop:ordering}.
  Suppose $\Law(v_i, \theta) \in \overline{\mathscr{P}}_G^{\mathrm{e}}(\mathcal{A}_1)$ for every $i \in \{1, \ldots, N\}$.
  Then, if $\mathbf{a} \coloneqq \E \mathbf{v}(X)$, with probability $1$ we have
  \begin{equation*}
    \sgn\big(v_i(x) - v_j(x)\big) = \sgn(a_i - a_j)
  \end{equation*}
  for all $x \in \R$ and every $i,j \in \{1, \ldots, N\}$.
\end{cor}
\begin{proof}
  Again, it suffices to consider the case $N = 2$.
  Our argument follows the proof of \cite[Proposition~6.1]{DGR21}.
  By \cref{prop:ordering}, with probability $1$ the random variable
  \begin{equation*}
    \chi \coloneqq \sgn(v_1(x) - v_2(x))
  \end{equation*}
  does not depend on $x \in \R$.
  Given $b \in \{0, \pm 1\}$ and $i \in \{1, 2\}$, define
  \begin{equation*}
    \mu_{i,b} \coloneqq \Law((v_i, \theta) \mid \chi = b)
  \end{equation*}
  if $\P(\chi = b) > 0$.
  Otherwise, let $\mu_{i,b} \coloneqq \Law(v_i, \theta)$.
  By the comparison principle, $\mu_{i,b}$ is time-invariant.
  Moreover, it is $G$-invariant because $\mu$ and $b$ are $G$-invariant.
  Therefore, we know that $\mu_{i,b} \in \overline{\mathscr{P}}_G(\mathcal{A}_N)$.
  Now, we can write $\Law(v_i,\theta)$ as the convex combination
  \begin{equation*}
    \Law(v_i,\theta) = \sum_{b \in \{0, \pm 1\}} \P(\chi = b) \mu_{i,b}.
  \end{equation*}
  Since $\Law(v_i,\theta)$ is extremal, we have $\mu_{i,b} = \Law(v_i,\theta)$ for all $b \in \{0, \pm 1\}$.
  Therefore, if~$X \sim \Uniform(\R/G)$ is independent of all else and $\P(\chi = b) > 0$, we have
  \begin{align*}
    \E\left[v_1(X) - v_2(X) \mid \chi = b\right] &= \E\left[v_1(X) \mid \chi = b\right] - \E \left[v_2(X) \mid \chi = b\right]\\
                                                 &= \E v_1(X) - \E v_2(X) = a_1 - a_2.
  \end{align*}
  Because $\chi$ does not depend on $x$, this implies that
  \begin{equation*}
    b = \E[\sgn\big(v_1(X) - v_2(X)\big)\mid \chi = b] = \sgn\big(\E\left[v_1(X) - v_2(X) \mid \chi = b\right]\big) = \sgn(a_1 - a_2).
  \end{equation*}
  Therefore $\P(\chi \neq \sgn(a_1 - a_2)) = 0$, as desired.
\end{proof}

\subsubsection*{The proof of \cref{thm:stochordering}}

Let $G$ be $\R$ or $L\Z$ for some $L > 0$.  For each $i \in \{1, 2\}$, fix $N_i \in \N$, let $N \coloneqq N_1 + N_2$, and 
take~$\mu_i \in \overline{\mathscr{P}}_G^{\mathrm{e}}(\mathcal{A}_{N_i})$.
Assume that $(v_{i;1}, \ldots, v_{i; N_i}, \theta_i) \sim \mu_i$ satisfies \cref{eq:momentcondition}.
By \cref{prop:cancouple}, there exists a coupling $\mu \in \mathscr{P}_G(\mathcal{A}_{N})$ of $\mu_1$ and $\mu_2$ in the sense of \cref{eq:isacoupling}.
Using \cref{prop:ordering}, we will show that $\mu$ is extremal, i.e.\ that $\mu \in \overline{\mathscr{P}}_G^{\mathrm{e}}(\mathcal{A}_{N})$.

Suppose there exist $\mu^{(0)}, \mu^{(1)} \in \mathscr{P}_G(\mathcal{A}_{N})$ and $\gamma \in (0, 1)$ such that
\begin{equation*}
  \mu = \gamma \mu^{(0)} + (1 - \gamma) \mu^{(1)}.
\end{equation*}
Then, if $(\mathbf{v},\theta) \sim \mu$ and $(\mathbf{v}^{(\ell)}, \theta^{(\ell)}) \sim \mu^{(\ell)}$ for $\ell \in \{0, 1\}$, \cref{eq:isacoupling} and \cref{eq:momentcondition} imply that
\begin{equation*}
  \E H(v_{i;j}^{(\ell)}(X)), \, \E (\partial_x v_{i;j}^{(\ell)}(X))^2 \leq \max\{\gamma^{-1}, (1 - \gamma)^{-1}\} \max\{\E H(v_{i;j}(X)), \E (\partial_x v_{i;j}(X))^2\} < \infty
\end{equation*}
for all $i \in \{1,2\}$ and $j\in \{1, \ldots, N_i\}$.
That is, $\mu^{(0)}$ and $\mu^{(1)}$ satisfy the hypotheses of \cref{prop:cancouple}.
Thus there exists $\hat{\mu} \in \overline{\mathscr{P}}_G(\mathcal{A}_{2N})$ such that if $((\mathbf{v}^{(0)}, \mathbf{v}^{(1)}), \theta) \sim \hat{\mu}$, then
\begin{equation*}
  \Law(\mathbf{v}^{(\ell)}, \theta) = \mu^{(\ell)} \quad \text{for each } \ell \in \{0, 1\}.
\end{equation*}

Fix $i \in \{1, 2\}$ and $j \in \{1, \ldots, N_i\}$.
We claim that the marginal $\mu_{i;j} \coloneqq \Law(v_{i;j}, \theta)$ of $\mu_i$ is extremal.
If it were a nontrivial convex combination of measures in $\overline{\mathscr{P}}_G(\mathcal{A}_1)$, we could use \cref{prop:cancouple} to couple \emph{those} measures to the remaining components of $\mu_i$ and thus write $\mu_i$ as a nontrivial convex combination.
It follows that
\begin{equation*}
  v_{i;j}^{(\ell)} \sim \mu_{i;j} \in \overline{\mathscr{P}}_G^{\mathrm{e}}(\mathcal{A}_1),
\end{equation*}
for each $\ell \in \{0, 1\}$.
By \cref{cor:extremal-ordering}, we have $v_{i;j}^{(0)} = v_{i; j}^{(1)}$ almost surely.
Since this holds for all $i,j$, we  have $\mathbf{v}^{(0)} = \mathbf{v}^{(1)}$ almost surely.
In particular, $\mu^{(0)} = \mu^{(1)}$.
Therefore, $\mu$ is extremal.

Finally, the extremality of the marginals $\Law(v_{i;j},\theta)$ and \cref{cor:extremal-ordering} imply \cref{eq:ordering}.
\hfill$\Box$
\medskip

Now, \cref{cor:uniqueness} follows from \cref{thm:stochordering}.
\begin{proof}[Proof of \cref{cor:uniqueness}]
  Take $N$, $G$, and $X$ as in the statement of \cref{cor:uniqueness}.
  Fix $\mathbf{a} \in \R^N$, and suppose there are two measures $\mu_1,\mu_2 \in \overline{\mathscr{P}}_G^{\mathrm{e}}(\mathcal{A}_N)$ with $(\mathbf{v}_i, \theta_i) = (v_{i;1},\ldots,v_{i;N},\theta_i)\sim \mu_i$ such that $\E \mathbf{v}_i(X) = \mathbf{a}$ and
  \begin{equation*}
    \E H\big(v_{i;j}(X)\big), \, \E \big(\partial_x v_{i;j}(X)\big)^2 < \infty \quad \text{for each } i \in \{1, 2\}, \, j \in \{1, \ldots, N\}.
  \end{equation*}
  By \cref{thm:stochordering}, there exists a coupling $\overline{\mathscr{P}}_G^{\mathrm{e}}(\mathcal{A}_{2N})$ satisfying \cref{eq:isacoupling} and \cref{eq:ordering}.
  In particular,
  \begin{equation*}
    0 = \sgn\big(\E v_{1;j}(X) - \E v_{2;j}(X)\big) = \chi_{1;j, 2;j} \quad \text{for every } j \in \{1, \ldots, N\}.
  \end{equation*}
  So $\mathbf{v}_1 = \mathbf{v}_2$ almost surely, and $\mu_1 = \mu_2$.
\end{proof}

\bibliographystyle{hplain-ajd}
\bibliography{burgers}

\end{document}